\documentclass[reqno]{amsart} 
\usepackage{amsfonts, amsmath, amsthm, amssymb, latexsym, graphicx}

\def\cal{\mathcal }
\theoremstyle{plain}
\newtheorem{theorem}{Theorem}[section]
\newtheorem{corollary}[theorem]{Corollary}
\newtheorem{lemma}[theorem]{Lemma}
\newtheorem{proposition}[theorem]{Proposition}
\theoremstyle{definition}

\newtheorem{remark}[theorem]{Remark}

\numberwithin{equation}{section}

\title[The differential equations of frozen CMS  models]{On the differential equations of frozen Calogero-Moser-Sutherland particle models}

\author{Michael Voit} 
\address{Fakult\"at Mathematik, Technische Universit\"at Dortmund,
          Vogelpothsweg 87,
          D-44221 Dortmund, Germany}
\email{michael.voit@math.tu-dortmund.de}

\begin{document}

\begin{abstract}
Multivariate Bessel and Jacobi processes   describe Calogero-Moser-Sutherland  particle models. 
They depend on a parameter $k$ and are related to time-dependent classical random matrix models like Dysom Brownian motions,
where $k$ has the interpretation of an inverse temperature.
There are several stochastic limit theorems for $k\to\infty$ were the limits depend on the solutions 
of  associated ODEs where these ODEs admit particular simple solutions which are connected with the zeros of the classical orthogonal polynomials.
In this paper we show that these solutions attract all solutions. Moreover we present a connection between
the solutions of these ODEs with associated inverse heat equations. These inverse heat equations are used to compute the
expectations  of some determinantal formulas for the  Bessel and Jacobi processes.
\end{abstract}

\date{\today}

\maketitle

\section{Introduction}

Consider Calogero-Moser-Sutherland particle models with $N$ ordered
particles on $\mathbb R$ or an suitable interval $I\subset\mathbb R$
where the positions of the particles are described by diffusions $(X_t=(X_t^1,\ldots,X_t^N))_{t\ge0}$ on
$$C_N:=\{x=(x_1,\ldots,x_N)\in I^N: \> x_1\le x_2\le\ldots \le x_N\}$$
which satisfy some stochastic differential equations (SDEs) depending on one or several parameters;
see e.g.~\cite{BF, DV, F, Ka, LV}.
We here mainly consider the following examples
 for some $N$-dimensional Brownian motion $(B_t)_{t\ge0}$ and $i=1,\ldots,N$:
\begin{align}\label{SDE-intro}
dX_{t}^i &= d B_t^i+ k\sum_{j: \> j\ne i} \frac{1}{X_{t}^i-X_{t}^j}dt \\
dX_{t}^i &= d B_{t,i}+ \beta \sum_{j: j\ne i}
\Bigl(\frac{1}{X_{t}^i-X_{t}^j}+
\frac{1}{X_{t}^i+X_{t}^j}\Bigr)dt
+ \frac{\nu\cdot\beta}{X_{t}^i}dt\notag\\
  dX_{t}^i & =\sqrt{2(1-(X_{t}^i)^2)}\> d B_{t,i} +
\kappa\Bigl((p-q) -(p+q)X_{t}^i +
2\sum_{j: \> j\ne i}\frac{1-X_{t}^iX_{t}^j}{X_{t}^i-X_{t}^j}\Bigr)dt
\notag\end{align}
with $I=\mathbb R, [0,\infty[, [-1,1]$ and $k,\beta,\nu, \kappa,p,q>0$ respectively where the singular boundaries are reflecting.
  Here and in future in all sums, $j:j\ne i$ means that we sum $j=1,\ldots,N, j\ne i$, where the sum is empty for $N=1$.

  It is 
  known that all SDEs admit strong unique solutions  even
  for starting points on the singular boundaries of $C_N$ for parameters  large enough;
  see \cite{AGZ, CGY, CL, De, GM, Sch1} for different approaches.
  Moreover, under even stronger conditions on the parameters,
  the solutions are in the interior of $C_N$ almost surely for $t>0$.

  In the 3 models above, the parameters $k,\beta,\kappa$ respectively have the interpretation of an inverse temperature.
  In the first case,    $(X_t)_{t\ge0}$ is just a Dyson Brownian motion (see e.g. \cite{AGZ}),
  the second case describes dynamic versions of $\beta$-Laguerre ensembles,
  and in the third case, $\beta$-Jacobi processes on compact alcoves appear
  with $\beta$-Jacobi ensemble-distributons as stationary distributions (with suitable parameters); see e.g. \cite{De, RR}.
  In all  cases, one can study freezing limits  for $k,\beta,\kappa\to \infty$
  respectively. This can be done  via the explicit transition probabilities of these processes.
  In particular, in the first two cases, the processes are Bessel processes associated with the root systems $A_{N-1}$ and $B_N$,
  where these densities can be described in terms of multivariate Bessel functions; see \cite{R1, RV1, V3} and references there.
  For weak limits and central limit theorems (CLTs) in the freezing case see
  \cite{AKM1, AKM2,AV1, AV2, CGY, DE2, R1, R2, RV1, V, V3}.

 In \cite{AV1, VW1}, an approach  to freezing CLTs via SDEs is given by using the renormalzed processes
  $(\tilde X_t:= X_{t/\tau})_{t\ge0}$  with $\tau:=k,\beta,\kappa$ respectively. In the SDEx for  $(\tilde X_t)_{t\ge0}$ then
   $k,\beta,\kappa$  disappear in the drift parts while in the Brownian parts an additional factor $1/\sqrt \tau$
  appears. Hence, in the freezing limit $\tau=\infty$, the SDEs  (\ref{SDE-intro}) degenerate into the 
  ordinary differential equations (ODEs)
\begin{align}\label{ODE-intro}
  \frac{dx_i(t)}{dt}&= \sum_{j:j\ne i} \frac{1}{x_i(t)-x_j(t)} ,\\
\frac{dx_i(t)}{dt}&= \sum_{j: j\ne i}\Bigl( \frac{1}{x_i-x_j}+  
\frac{1}{x_i+x_j}\Bigr)+\frac{\nu}{x_i}, \notag\\
  \frac{dx_i(t)}{dt}&= (p-q)-(p+q)x_i(t)
			+2\sum_{j:\> j\neq i}\frac{1-x_i(t)
			  x_j(t)}{x_i(t)-x_j(t)},\notag
\end{align}
for  $\nu, p,q>0$.
In all  cases these ODEs admit unique solutions in the interior of $C_N$ for $t>0$ even for starting points on the boundary
similar to the solutions of the SDEs in (\ref{SDE-intro}) above; see E.G.~\cite{VW3, AVW}. Moreover, as the solutions $x(t)$ of the ODEs (\ref{ODE-intro})
appear in the means and covariances of freezing CLTs for $\tilde X_t$ (see \cite{VW1}), one is interested in properties of the solutions
$x(t)$ of the ODEs (\ref{ODE-intro}).
It is not possible to determine the solutions  of the ODEs (\ref{ODE-intro}) in general explicitely.
On the other hand, in all 3 cases there are remarkable special solutions which can be described in terms of the ordered zeros of the
three classical  orthogonal polynomials of order $N$, namely Hermite, Laguerre, and Jacobi polynomials,
where the parameters in the Laguerre and Jacobi case depend on  $\nu, p,q>0$ respectively. This connection follows from some potential-theoretic
interepretation of these zeros due to Stieltjes (see Section 6.7 of \cite{S})
and was used in the context above e.g.~in \cite{AKM1, AKM2, AV1, AV2, HV, V, VW1}.
Moreover, a lot of concrete details for symmetric polynomials of the general solutions 
$x(t)$ of the ODEs (\ref{ODE-intro}) are known; see e.g.~\cite{AVW, KVW, RV2, V2, VW2, VW3} are known. In particular, recursive formulas and limit results for
$N\to\infty$ for $x_1(t)^n+\cdots+x_N(t)^n$ ($n\ge0$) are used in \cite{AVW, VW2} to derive limit theorems related to free probability.

The aim of this paper is to investigate the general solutions of the ODEs (\ref{ODE-intro}) beyond the papers mentioned above.
The focus will be  on two topics. First of all, we  show in all  cases that for $t\to\infty$,
arbitrary solutions $x(t)$ always tend to the special explicit solutions 
$x_0(t)$ which are related to the zeros of the classical  orthogonal polynomials  with explicit estimates for the errors.
For these estimates we  present several proofs; some will be based  on explicit computations and the Lemma of Gronwall while other ones
are based on some reformulations of the problems, the stability of equilibria, and properties of gradient systems; see Ch. 9 of \cite{HS}.

The second main topic of the paper is the  observation that a differentiable function $x(t)$ on $C_N$
solves the first ODE in (\ref{ODE-intro})
      if and only if $H(t,z):=\prod_{i=1}^N (z-x_i(t))$
      solves the inverse heat equation $H_t+ \frac{1}{2}H_{zz}=0$.
      This connection appears for $N=\infty$  in \cite{CSL, RT} in the context of the Riemannian hypothesis. To explain this we recapitulate
(see \cite{CSL, RT} and references there) that  the Riemannian hypothesis is equivalent to the fact that some 
      concrete complicated function $H$, which satisfies the inverse heat equation $H_t+ \frac{1}{2}H_{zz}=0$, has the property
      that the infinitely many zeros of $H(0,z)$ are real-valued. It is shown in \cite{CSL} that under  technical conditions,
      a solution $H$ of $H_t+ \frac{1}{2}H_{zz}=0$ has the following property: if $H$ has only  real-valued zeros for some $t_0\in\mathbb R$, then it has
      only   real-valued simple  zeros $x_i(t)$ for all $t>t_0$, which we may order by size, where these zeros the satisfy the first ODE in (\ref{ODE-intro})
      (for $N=\infty$). Clearly, this stamenet is closely related to the if-and-only-if statement for finite $N$ above, and to the  unique solvability
      of the first ODE ins (\ref{ODE-intro}) even for starting points on the boundary from \cite{VW3}. As this connection between
      the first ODE in (\ref{ODE-intro}) and the  inverse heat equation also appears in the context of heat flows for Gaussian analytic functions
      in \cite{HHJK}, it is an interesting task whether there exist such connections for the other ODEs in  (\ref{ODE-intro}) and associated
       inverse heat equations. We shall show that such results are in fact available.  In particular,  the second  ODE in  (\ref{ODE-intro})
      is connected to the inverse Bessel heat 
 equation
       $$ \partial_tH = -\Bigl(\frac{1}{2}\partial_{zz}+ \frac{\nu-1/2}{z}\partial_z\Bigr)H=-G_{\nu-1}H,$$
and  the third  ODE in  (\ref{ODE-intro})
to the inverse Jacobi-type heat equation with potential
 $$  H_t=  -\Biggl( (1-z^2) H_{zz}+ \Bigl( (p-q)+(2(N-1)-(p+q))z\Bigr) H_{z} \Biggr) -N(p+q-N+1)H. $$
We shall use these relations to derive explicit formulas for expectations of the form
$$\mathbb E\bigl(\prod_{i=1}^N (Y_t- \tilde X_{t}^i)\bigr)$$
for the processes $(\tilde X_t)_{t\ge0}$ above and a stochastically independent process $(Y_t)_{t\ge0}$, which, depending on the ODE in 
(\ref{ODE-intro}), is  a Brownian motion on $\mathbb R$, a one-dimensional Bessel process on $[0,\infty[$, or a Jacobi process on $[-1,1]$
    respectively.
    Besides the three cases  above we also consider an hyperbolic analogue of the compact Jacobi case with $N$ ordered particles in $[1,\infty[$,
        in which case the processes are known as Heckman-Opdam processes of type BC; see \cite{Sch1,Sch2}. By the close algebraic connection 
of both cases, we easily obtain a connection to some inverse heat equation also in this case.

The paper is organized as follows. Section 2 is  devoted to the Hermite case, i.e., the first case in  (\ref{ODE-intro}) and 
(\ref{SDE-intro}). The Laguerre case, i.e., the sercond case, will be treated in Section 3.
Section 4 and 5 then contain the compact and non-compact Jacobi case. Finally, in Section 6 we study a particle model 
 on the torus $\mathbb T:=\{z\in\mathbb C:\> |z|=1\}$ where some details are  different from the other cases; see \cite{HW, LV, RV2}.

We  always denote the usual euclidean norm on $\mathbb R^N$ by $\|.\|$.

\section{The Hermite case}

In this section we study the ODE
\begin{equation}\label{ODE-a}
x_i^\prime(t)= \sum_{j\ne i} \frac{1}{x_i(t)-x_j(t)} \quad\quad (i=1,\ldots,N)
\end{equation}
for $N\ge2$. We  recapitulate some known facts.

At the beginning, the ODE is considered 
on the interior $ W_N^A:=\{x\in\mathbb R^N: \> x_1<\ldots<x_N\}$ of   $ C_N^A$.
It is shown in \cite{VW3} that it can be also uniquely solved for starting points in the boundary:

\begin{theorem}\label{ode-ex-unique-a-thm}
For each $x(0)\in C_N^A$,  (\ref{ODE-a})
has a unique solution for  $t\ge0$ in  $ C_N^A$, i.e.,
 there is a unique continuous function
$x:[0,\infty[\to C_N^A$ with  $x(t)\in  W_N^{A} $  for $t\in]0,\infty[$
which satisfies  (\ref{ODE-a}) for $t\in]0,\infty[$.
\end{theorem}

The ODEs
(\ref{ODE-a}) have particular simple solutions which can be expressed in terms of the  zeros of
 Hermite polynomials. For this consider the  Hermite polynomials $(H_n)_{n\ge 0}$  which are orthogonal w.r.t.
 the density  $e^{-x^2}$ as for instance in Section 5.5 of \cite{S}.
 
 A classical result of Stieltjes gives the following characterization of the zeros of $H_N$; see Section 6.7 of \cite{S} and also
  \cite{AKM1,AV1, V}:

\begin{lemma}\label{char-zero-A}
 For $ {\bf z}=(z_1,\ldots,z_N)\in C_N^A$, the following  are equivalent:
\begin{enumerate}
\item[\rm{(1)}]  $x\mapsto W(x):=\sum_{i,j: i<j} \ln(x_i-x_j) -\|x\|^2/2$ is maximal in ${\bf z}$;
\item[\rm{(2)}] For $i=1,\ldots,N$,  $z_i= \sum_{j: j\ne i} \frac{1}{z_i-z_j}$;
\item[\rm{(3)}] $z_{1}<\ldots< z_{N}$ are
 the ordered zeros  of the Hermite polynomial $H_N$.
\end{enumerate}
\end{lemma}

This lemma immediately implies the following; see  \cite{VW1, VW2, VW3}:

\begin{lemma}\label{special-solution-a}
 Let ${\bf z}:=(z_1,\ldots,z_N)\in  W_N^A$ be the vector consisting
of the ordered zeros of  $H_N$.
Then  $x(t):=\sqrt{2t}\cdot {\bf z}$ is a solution of  (\ref{ODE-a}).
\end{lemma}

We next turn to general solutions of (\ref{ODE-a}).
The following facts can be  easily verified; see \cite{VW3}.

\begin{lemma}\label{symmetric-pol-pol-in-t}
\begin{enumerate}
\item[\rm{(1)}] If $x(t)$ is a solution of  (\ref{ODE-a}), then for each constant $c\in\mathbb R$,
  $x(t)+c\cdot(1,\ldots,1)$ and $x(t+c)$ are also  solutions of  (\ref{ODE-a})
  (where in the second case,  $t\in [-c,\infty[$). Moreover, for $c>0$,  $\frac{1}{\sqrt c} x(ct)$ is also a
    solution of  (\ref{ODE-a}).
\item[\rm{(2)}] For each solution $x(t)$ of  (\ref{ODE-a}) with start in  $ C_N^A$,
\begin{equation}\label{ODE-norm-a}
\|x(t)\|^2 =N(N-1)t+\|x(0)\|^2.
\end{equation}
\end{enumerate}
\end{lemma}

Lemma \ref{special-solution-a} and (\ref{ODE-norm-a})  imply that
\begin{equation}\label{ODE-norm-hermite-zeros}
  z_1^2+\ldots+z_N^2=N(N-1)/2;
\end{equation}
see also  \cite{AKM1}.

We shall prove that for $t\to\infty$, all solutions of  (\ref{ODE-a}) tend to the particular solution 
in Lemma \ref{special-solution-a} up to  transformations from Lemma \ref{symmetric-pol-pol-in-t}(1).
Before doing so, we consider the euclidean difference between two solutions of (\ref{ODE-a}):

 \begin{lemma}\label{decreasing-error-a}
For all solutions $x(t),\tilde x(t)$ of  (\ref{ODE-a}),  $\|x(t)-\tilde x(t)\|$ is decreasing.
 \end{lemma}

 \begin{proof}
   Let $r_i(t):=x_i(t)-\tilde x_i(t)$ ($i=1,\ldots,N$) and $D(t):=\frac{1}{2}\sum_{i=1}^N r_i(t)^2$. Then
   \begin{align}\label{error-est-a}
     D^\prime(t) &= \sum_{i=1}^N\Biggl( x_i^\prime(t)x_i(t) +\tilde x_i^\prime(t)\tilde x_i(t)- {\tilde x_i}^\prime(t)x_i(t)- x_i^\prime(t)\tilde x_i(t)
     \Biggr)\notag\\
     &= \sum_{i,j: i\ne j}\Biggl(\frac{x_i(t)}{x_i(t)-x_j(t)}+\frac{\tilde x_i(t)}{\tilde x_i(t)-\tilde x_j(t)}-
     \frac{ x_i(t)}{\tilde x_i(t)-\tilde x_j(t)}-\frac{\tilde x_i(t)}{ x_i(t)- x_j(t)} \Biggr)\notag\\
     &= \sum_{1\le i< j\le N} \Biggl(2-  \frac{ x_j(t) - x_i(t)}{\tilde x_j(t)-\tilde x_i(t)}- \frac{\tilde x_j(t) -\tilde x_i(t)}{x_j(t)- x_i(t)}
     \Biggr)\notag\\
     &\le0,
   \end{align}
as claimed by the ordering of the components, as for $a>0$, $a+\frac{1}{a}-2\ge 0$ holds.
   \end{proof}

 Please notice that  $\|x(t)-\tilde x(t)\|$ is constant for any  solution $x(t)$ and the shifted solution
$\tilde x(t):=x(t)+c\cdot(1,\ldots,1)$ with  $c\ne 0$ by Lemma \ref{symmetric-pol-pol-in-t}(1).
For this reason we  now  only compare  solutions  $x(t),\tilde x(t)$ with $x_1(0)+\ldots+ x_N(0)=\tilde x_1(0)+\ldots+ \tilde x_N(0)$
without loss of generality.

We next consider the angular parts  of   solutions  $x(t)$ of (\ref{ODE-a}).
A direct computation yields the following; see Lemma 2.4 of \cite{VW3} and its proof there.

\begin{lemma}\label{stabilily-lemma1}
  For each initial value $x(0)\in W_N^A $, consider the solution $x(t)$  of (\ref{ODE-a}) and its angular part
  $\phi(t):=x(t)/\|x(t)\|$.
  Then the time-transformed angular part
  \begin{equation}\label{time-transform-a}
    \psi(t) := \phi\Bigl( \frac{ N(N-1)}{2}t^2+ \|x(0)\|^2t \Bigr) \quad(t\ge0)
    \end{equation}
satisfies
\begin{equation}\label{ODE-tilde-a}
{\psi}_i^\prime(t) = \sum_{j: j\ne i}  \frac{1}{\psi_{i}(t)-\psi_{j}(t)}
-\frac{ N(N-1)}{2} \psi_{i}(t)\quad(i=1,\ldots,N)
\end{equation}
with $\psi(0)=\phi(0)=x(0)/\|x(0)\|$.
\end{lemma}

Notice that by the very construction in Lemma \ref{stabilily-lemma1},
for each starting point $\psi(0)\in S^{N-1}\cap  W_N^A$ on the unit sphere
$ S^{N-1}:=\{x\in\mathbb R^N:\> \|x\|=1\}$, the solution of (\ref{ODE-tilde-a}) satisfies 
$\psi(t)\in S^{N-1}$ for  $t\ge0$. This can be also derived from (\ref{ODE-tilde-a}).

We  remark that $\sqrt{\frac{2}{N(N-1)}}\cdot  {\bf z }\in  S^{N-1}$ is the only
equilibrium point of (\ref{ODE-tilde-a})
  in  $W_N^A$ by Lemma \ref{char-zero-A}. Moreover,
(\ref{ODE-tilde-a}) is a gradient system, which proves in combination with standard results on gradient systems
(see e.g. Section 9.4 of \cite{HS}) that each solution of (\ref{ODE-tilde-a})
satsfies
$$\lim_{t\to\infty} \psi(t) = \sqrt{\frac{2}{N(N-1)}}\cdot  {\bf z }\in  S^{N-1};$$
see \cite{VW3}.
We shall improve this result below.

Before doing so, we collect some further properties of the ODE (\ref{ODE-tilde-a})-
For this we first recapitulate a further connection between solutions (\ref{ODE-tilde-a}) and  (\ref{ODE-a})
which was already observed in \cite{VW1, VW2}; see in particular Eqs. (2.26) and  (2.27) in \cite{VW2},
and which can be checked easily.
In probability, a correponding connection is well-known and gives
  a pathwise connection between  classical Brownian motions and classical Ornstein-Uhlenbeck processes, i.e.,
  asymptotically stationary versions  of Brownian motions.

  \begin{lemma}\label{connection-odes2-a}
Let $\lambda>0$ and $x_0\in C_N^A$. Let $x(t,x_0)$ denote the  solution of  (\ref{ODE-a}) with start in $x_0$.
Then for $t\ge0$,
   \begin{equation}\label{renorming-stationary}
\hat x(t,x_0):=      x(\frac{1-e^{-2\lambda t}}{2\lambda}, e^{-\lambda t}x_0)
\end{equation}
is a solution of 
\begin{equation}\label{renorming-stationary-dgl}
  \hat x_i^\prime(t) =\sum_{j\ne1} \frac{1}{ \hat x_1(t)- \hat x_j(t)} 
-\lambda \hat  x_i(t) \quad (i=1,\ldots,N) \quad\text{with}\quad  \hat x(0)=x_0 .
\end{equation}
Moreover, the converse statement is also true, i.e., if (\ref{renorming-stationary-dgl}) holds for $t\ge0$, then
(\ref{ODE-a}) holds for $t\in[0,1/(2\lambda)[$.
  \end{lemma}

By this lemma we in particular can extend Theorem \ref{ode-ex-unique-a-thm} to the ODE  (\ref{ODE-tilde-a}):

\begin{theorem}\label{ode-ex-unique-a-thm-stat}
For each $\psi(0)\in C_N^A$,  (\ref{ODE-tilde-a})
has a unique solution $\psi(t)$ for  $t\ge0$ in  $ C_N^A$ in the sense as described in Theorem \ref{ode-ex-unique-a-thm}.
\end{theorem}

We next apply the method of the proof of Lemma \ref{decreasing-error-a} to the ODE  (\ref{ODE-tilde-a}).
In this case we obtain a stronger result:

\begin{proposition}\label{decreasing-error-a-stat}
  Let  $\psi(t),\tilde \psi(t)$ be solutions of (\ref{ODE-tilde-a}) with $\psi(0),\tilde \psi(0)\in C_N^A$.
Then for $t\ge0$,
  $$\Bigl\|\psi(t)-\tilde \psi(t)\Bigr\|\le  e^{- \frac{N(N-1)}{2} t}\Bigl\|\psi(0)-\tilde \psi(0)\Bigr\|.$$ 
 \end{proposition}

 \begin{proof}
   Let $r_i(t):=\psi_i(t)-\tilde \psi_i(t)$ ($i=1,\ldots,N$) and $D(t):=\frac{1}{2}\sum_{i=1}^N r_i(t)^2$. Then, as
   in (\ref{error-est-a}),
\begin{align}
     D^\prime(t) = &
     \sum_{1\le i< j\le N} \Biggl(2-  \frac{ \psi_j(t) - \psi_i(t)}{\tilde \psi_j(t)-\tilde \psi_i(t)}-
     \frac{\tilde \psi_j(t) -\tilde \psi_i(t)}{\psi_j(t)- \psi_i(t)}
     \Biggr)\notag\\
&- \frac{N(N-1)}{2}  \sum_{1=1}^N (\psi_i(t)^2 +\tilde\psi_i(t)^2-2\psi_i(t)\tilde\psi_i(t))\notag\\
     \le& - \frac{N(N-1)}{2} \sum_{1=1}^N (\psi_i(t)-\tilde\psi_i(t))^2 = -N(N-1) D(t).
\notag
   \end{align}
The lemma of Gronwall now implies the claim.
   \end{proof}

 \begin{remark}\label{other-proofs}
   Proposition \ref{decreasing-error-a-stat} (or at least slightly weaker statements) can be derived also by other 
  methods which may be also interesting. However, the preceding proof seems to be the easiest one. We sketch two other approaches:
\begin{enumerate}
\item[\rm{(1)}]  By Lemma \ref{char-zero-A},
  $\sqrt{\frac{2}{N(N-1)}}\cdot  {\bf z }$ is the only equilibrium point of (\ref{ODE-tilde-a})
  in  $C_N^A$. Moreover,  by an easy computation, the  Jacobi matrix $J$ of the right hand side of (\ref{ODE-tilde-a}) at this equilibrium  point is 
  $$J= \Bigl(1-\frac{N(N-1)}{2}\Bigr)\cdot I_N - S_N$$
  with the $N$-dimensional identity  $I_N$ and the matrix $S_N=(s_{i,j})_{i,j=1}^N$ with
\begin{equation}\label{covariance-matrix-A}
s_{i,j}:=\left\{ \begin{array}{r@{\quad\quad}l}  1+\sum_{l\ne i} (z_{i}-z_{l})^{-2} & \text{for}\quad i=j \\
   -(z_{i}-z_{j})^{-2} & \text{for}\quad i\ne j  \end{array}  \right.  . 
\end{equation}
The matrix $S_N$ appears in \cite{V, AV2} as the inverse of the covariance matrix of some $N$-dimensional
freezing central limit theorem (CLT) for $\beta$-Hermite ensembles and Dyson Brownian motions when the inverse temperature $\beta$ tends to $\infty$;
see also
 Dumitriu and Edelman \cite{DE2} for this CLT.
It is shown in \cite{AV2} that $S_N$ 
has the eigenvalues $1,2,\ldots,N$. This means that $-\frac{N(N-1)}{2}$ is the largest eigenvalue of the symmetric matrix $J$.
This leads to a slightly weaker version of Proposition \ref{decreasing-error-a-stat}
by a standard result on the stability of equilibria;
see for instance Section 9.1 of \cite{HS}.
\item[\rm{(2)}] A combination of Lemmas \ref{connection-odes2-a}
and  \ref{decreasing-error-a} leads to a third proof  of Proposition \ref{decreasing-error-a-stat}.
\end{enumerate}
\end{remark}

 Proposition \ref{decreasing-error-a-stat} and 
Lemma \ref{stabilily-lemma1} lead to the following  convergence for solutions of  (\ref{ODE-a}):

\begin{theorem}\label{final-convergence-a}
  Let  $x(t),\tilde x(t)$ be solutions of  (\ref{ODE-a}) on $C_N^A$ with
  $$x_1(0)+\ldots+ x_N(0)=\tilde x_1(0)+\ldots+ \tilde x_N(0) \quad\text{and}\quad
\|x(0)\|=\|\tilde x(0)\|>0$$
Then for all $t\ge0$,
$$\|x(t)-\tilde x(t)\|\le
\|x(0)-\tilde x(0)\|\cdot\sqrt{\frac{N(N-1)}{\|x(0)\|^2}t +1}\cdot  e^{- \frac{1}{2}  \Bigl(\sqrt{ 2N(N-1)t +\|x(0)\|^4}- \|x(0)\|^2\Bigr) }.$$

In particular, if  $x(t)$ is any solution   of   (\ref{ODE-a}) on $C_N^A$ with  $x_1(0)+\ldots+x_N(0)=0$, then
\begin{align}
  \Bigl\|x(t)&-\sqrt{2t+\frac{2\|x(0)\|^2}{N(N-1)}}\cdot  {\bf z }\Bigr\|\\
  \le&\Biggl\|x(0)-\sqrt{\frac{2\|x(0)\|^2}{N(N-1)}}\cdot  {\bf z }
\Biggr\|\cdot
\sqrt{\frac{N(N-1)}{\|x(0)\|^2}t +1}\cdot  e^{- \frac{1}{2}   \Bigl(\sqrt{ 2N(N-1)t +\|x(0)\|^4}- \|x(0)\|^2\Bigr) }
\notag\end{align}
\end{theorem}

\begin{proof} We use the time transform
  $$t=\frac{ N(N-1)}{2}\tau^2+ \|x(0)\|^2\tau$$
  from (\ref{time-transform-a}) with $t,\tau\ge 0$. Then
  $$\tau = \frac{1}{N(N-1)}\Bigl(\sqrt{ 2N(N-1)t +\|x(0)\|^4}- \|x(0)\|^2\Bigr),$$
  and by Proposition \ref{decreasing-error-a-stat} and Lemmas  \ref{stabilily-lemma1} and \ref{stabilily-lemma1},
  \begin{align}
    \|x(t)&-\tilde x(t)\|= \|x(t)\|\cdot
    \Bigl\| \frac{x(t)}{\|x(t)\|}- \frac{\tilde x(t)}{\|\tilde x(t)\|}\Bigr\|  \notag\\
  &\le\|x(t)\|\cdot \Bigl\| \frac{x(0)}{\|x(0)\|}- \frac{\tilde x(0)}{\|\tilde x(0)\|}\Bigr\|
 e^{- \frac{N(N-1)}{2}  \tau}\notag\\
&= \sqrt{\frac{N(N-1)}{\|x(0)\|^2}t +1}\cdot \|x(0)-\tilde x(0)\|\cdot  e^{- \frac{1}{2} \Bigl(\sqrt{ 2N(N-1)t+\|x(0)\|^4}- \|x(0)\|^2\Bigr) }.
 \notag \end{align}
This proves the first estimate. The second estimate is then also clear. 
  \end{proof}

We next discuss some connection  of the ODE (\ref{ODE-a}) to the inverse heat equation.
As indicated in the introduction, this is not new; see for instance \cite{CSL, RT} for a connection for $N=\infty$
in the context of the Riemannian hypothesis and also \cite{HHJK}.

We include a proof, as this will be the blueprint for
the proofs of corresponding results for further ODEs below.

\begin{theorem}\label{heat-equation-a}
  Let $x:=(x_1,\ldots,x_N):[0,\infty[\to C_N^A$ be a differentiable function.
      Then $x$ is a solution of (\ref{ODE-a}) in the sense of Theorem \ref{ode-ex-unique-a-thm}
      if and only if the function $H(t,z):=\prod_{i=1}^N (z-x_i(t))$
      solves the inverse heat equation $H_t+ \frac{1}{2}H_{zz}=0$.
\end{theorem}

\begin{proof}
  Assume first that $x(t)$ satisfies  (\ref{ODE-a}). Consider $H$ as defined in the theorem. For $t>0$ and $i,j=1,\ldots,N$ with $i\ne j$
  consider the polynomials $H_i$ and $H_{i,j}$ in $z$ with $H_i(t,z):=H(t,z)/(z-x_i(t))$ and  $H_{i,j}(t,z):=H(t,z)/((z-x_i(t))(z-x_j(t)))$.
  Then,  for $i\ne j$,
  \begin{equation}\label{two-to-one}
\frac{H_i(t,z)-H_j(t,z)}{x_i(t)-x_j(t)}=H_{i,j}(t,z).
    \end{equation}
  Hence, by (\ref{ODE-a}),
  \begin{align}
    \partial_t H(t,z)=&    -\sum_{i=1}^N x_i^\prime(t) \> H_i(t,z)=-\sum_{i,j: i\ne j} \frac{1}{x_i(t)-x_j(t)}\> H_i(t,z)\\
    =& - \frac{1}{2} \sum_{i,j: i\ne j }H_{i,j}(t,z) =  - \frac{1}{2}\partial_z\Bigl( \sum_{i=1}^N  H_i(t,z)\Bigr) =  - \frac{1}{2}\partial_{zz}  H(t,z)
\notag
\end{align}
  as claimed. Please notice that the factor $1/2$ in this computation comes from the fact that two terms  in the first line of (\ref{two-to-one})
   lead to one in the second line.

  Now assume that $H_t+ \frac{1}{2}H_{zz}=0$. As before, we write $H(t,z)=H_i(t,z)(z-x_i(t))$ for $i=1,\ldots,N$. Hence,
    \begin{equation}\label{log-derivative-a1}
      \frac{\partial_{zz}H(t,z)}{\partial_{z}H(t,z)}=
      \frac{2\partial_{z} H_i(t,z)+(z-x_i(t))\partial_{zz}H_i(t,z)}{H_i(t,z)+(z-x_i(t))\partial_{z}H_i(t,z)}
    \end{equation}
    and thus
     \begin{equation}\label{log-derivative-a2}
\frac{\partial_{zz}H(t,x_i(t))}{\partial_{z}H(t,x_i(t))}= \frac{2\partial_{z} H_i(t,x_i(t))}{H_i(t,x_i(t))}.
    \end{equation}
Moreover,
\begin{equation}\label{implicit} 0=\frac{d}{dt} H(t,x_i(t))= \partial_{t} H(t,x_i(t))+\partial_{z} H(t,x_i(t))\cdot x_i^\prime(t)
  \end{equation}
 and the heat equation lead to
 \begin{equation}\label{log-derivative-a3}
 x_i^\prime(t)=- \frac{\partial_{t} H(t,x_i(t))}{\partial_{z} H(t,x_i(t))} =  \frac{\partial_{zz} H(t,x_i(t))}{2\partial_{z} H(t,x_i(t))} .
 \end{equation}
 Hence, by (\ref{log-derivative-a2}),
 \begin{equation}\label{log-derivative-a4}
   x_i^\prime(t)=\frac{\partial_{z} H_i(t,x_i(t))}{H_i(t,x_i(t))}= \sum_{j:j\ne i} \frac{1}{x_i(t)-x_j(t)} \end{equation}
     as claimed.
\end{proof}

\begin{remark}\label{martingale-a}
  Solutions of the inverse heat equation $H_t+ \frac{1}{2}H_{zz}=0$ are known in probability theory as space-time harmonic functions.
   Polynomial solutions of this type have
   the property that for a  Brownian motion $(B_t)_{t\ge0}$ on $\mathbb R$ (with the generator $f\mapsto f^{\prime\prime}/2$
   as usual in probability), the process $(H(t, B_t))_{t\ge0}$ is a martingale w.r.t.~the canonical filtration
   by Dynkin's formula; see e.g. Section III.10 of \cite{RW}.
   
   The most typical examples of such  space-time-harmonic functions are the  heat polynomials
\begin{equation}\label{heat-pol-a}
   H_N(t,z):=2^N\prod_{i=1}^N (z-\sqrt{2t}\cdot z_i)=2^N (2t)^{N/2}\prod_{i=1}^N\Bigl( \frac{z}{\sqrt{2t}}-z_i\Bigr)=
   (2t)^{N/2} \cdot H_N(z/\sqrt{2t})
   \end{equation}
with the Hermite polynomials $H_N$ with the standard leading coefficients $2^N$ as in \cite{S} for $N\ge0$ where we again see the
connection with Lemma \ref{special-solution-a}

   Theorem \ref{heat-equation-a} and the linearity of the inverse heat equation  show
   that for each solution $x(t)$ of the ODE (\ref{ODE-a}) on $C_N^A$, the associated polynomial $H(t,z)$  from Theorem \ref{heat-equation-a}
   can be written  as
$$H(t,z)=2^{-N}H_N(t,z)+\sum_{l=0}^{N-1} c_lH_l(t,z),$$
   where  the $c_l$ ($l=0,\ldots, N-1$) are determined uniquely by the initial condition:
   $$2^Nz^N +\sum_{l=0}^{N-1} c_l 2^l z^l= \prod_{i=1}^N (z-x_i(0)).$$
\end{remark}

We next discuss some application of Theorem \ref{heat-equation-a} to  some expectations for Dyson Brownian motions
on the Weyl chambers $C_N^A$ for arbitrary parameters $k>0$ where $k$ describes some inverse temperature; see e.g.~\cite{AGZ}.
These processes $(X_{t,k}=(X_{t,k}^1,\ldots, X_{t,k}^N))_{t\ge0}$ are also known as Bessel processes on $C_N^A$; see e.g. \cite{CGY, RV1, V3}.
Following \cite{AGZ,  CGY}, we  can describe these processes as the unique strong solutions of the stochastic differential equations 
\begin{equation}\label{SDE-A}
 dX_{t,k}^i = d\tilde B_t^i+ k\sum_{j: \> j\ne i} \frac{1}{X_{t,k}^i-X_{t,k}^j}dt \quad\quad(i=1,\ldots,N)
\end{equation}
with some $N$-dimensional Brownian motion $(\tilde B_t)_{t\ge0}$ for arbitrary starting points $x\in C_N^A$ where 
 $x$ may be even on the singular boundary of $C_N^A$; see  Section 4.3 of \cite{AGZ} and \cite{CL, GM, CGY} for different approaches. 
The processes $(X_{t,k})_{t\ge0}$ have  the generators 
 \begin{equation}\label{generator-hermite-intro}
     L_kf:= \frac{1}{2} \Delta f +
 k \sum_{i=1}^N\Bigl( \sum_{j\ne i} \frac{1}{x_i-x_j}\Bigr) \frac{\partial}{\partial x_i}f 
 \end{equation}
 for
 $$f\in D(L_k):=\{f\in \cal C^{2}(\mathbb R^N), 
 \>\>\> f\>\>\text{ invariant under all permutations of coordinates}\}$$
 where the transition probabilities can be described in terms of multivariate Bessel functions; see \cite{R1, R2,  RV1, V3}.

 Clearly, the  renormalized processes $(\tilde X_{t,k}:=X_{t/k,k})_{t\ge0}$ are solutions of
\begin{equation}\label{SDE-A-normalized}
d\tilde X_{t,k}^i =\frac{1}{\sqrt k}dB_t^i + \sum_{j: \> j\ne i} 
 \frac{1}{\tilde X_{t,k}^i-\tilde X_{t,k}^j}dt\quad\quad(i=1,\ldots,N). 
\end{equation}
Moreover, for  $k=\infty$   (\ref{SDE-A-normalized}) degenerates into the ODE (\ref{ODE-a}) with the deterministic solutions $x(t)$.

In \cite{KVW}, several  martingales associated with the  $(X_{t,k})_{t\ge0}$
were considered which lead to some identities for expectations which generalize identities
for determinants of Gaussian orthogonal/unitary/symplectic ensembles in \cite{DG, FG}. The key in  \cite{KVW} is the observation that the expectations
of elementary symmetric polynomials  of $\tilde X_{t,k}$ are independent of $k\in]0,\infty[$. For this recapitulate
that  the elementary
symmetric polynomials $e_l^N(x)$ ($ l=0,\ldots,N$) in $N$ variables satisfy
\begin{equation}\label{symmetric-poly}
\prod_{l=1}^N (z-x_l) = \sum_{l=0}^{N}(-1)^{N-l}  e^N_{N-l}(x) z^l \quad\quad
(z\in\mathbb C, \> x=(x_1,\ldots,x_N)),
\end{equation}
i.e.,
$e_0^N=1, \> e_1^N(x)=\sum_{l=1}^N x_l , \ldots, e_N^N(x)=x_1x_2\cdots x_N$. Then, by Corollary 2.5 in \cite{KVW}:

\begin{lemma}\label{exp-independence-a} For all $ l=0,\ldots,N$, $t\ge0$, $k\in]0,\infty]$, and fixed starting points $x\in C_N^A$,
  the processes $(\tilde X_{t,k})_{t\ge0}$
   satisfy $\mathbb E( e_l^N(\tilde X_{t,k}))=  e_l^N(x(t))$.  
\end{lemma}

\begin{corollary}\label{exp-independence-a-cor} 
  For fixed starting points $x\in C_N^A$ and  $k\in]0,\infty]$, consider the Bessel processes $(\tilde X_{t,k})_{t\ge0}$. Then,  for
each $t\ge0$ and each $\mathbb R$-valued random variable $Y$ with $N$-th moment, which is independent from $\tilde X_{t,k}$,
$$\mathbb E\bigl(\prod_{i=1}^N (Y- \tilde X_{t,k}^i)\bigr)=\mathbb E\bigl(\prod_{i=1}^N (Y- x_l(t))\bigr).$$
\end{corollary}

\begin{proof}
\begin{align}\mathbb E\bigl(\prod_{i=1}^N (Y- \tilde X_{t,k}^i)\bigr)&=
\sum_{l=0}^N (-1)^l\mathbb E\bigl(e_l^N( \tilde X_{t,k}) Y^{N-l}\bigr)
=\sum_{l=0}^N (-1)^l\mathbb E(e_l^N(\tilde X_{t,k})\bigr) \cdot\mathbb E(Y^{N-l})\notag\\
&=\sum_{l=0}^N (-1)^l e_l^N( x(t)) \cdot\mathbb E(Y^{N-l})
=\mathbb E\bigl(\prod_{l=1}^N (Y- x_l(t))\bigr) .
\notag\end{align}
\end{proof}

For $Y=z\in\mathbb R$ a constant and the starting point $x=0\in C_N^A$, this, Lemma \ref{special-solution-a}, and the definition of
the Hermite polynomial $H_N$ yield the following; see \cite{KVW}:

\begin{corollary}\label{exp-independence-a-cor1}
  For $k>0$ let
$( X_{t,k})_{t\ge0}$ be the (original) Bessel process of type A with  start in $0$.
Then, for $t>0$ and $z\in\mathbb R$,
$$\mathbb E\bigl(\prod_{i=1}^N (z- X_{t,k}^i)\bigr)    = (tk/2)^{N/2}\cdot H_N(z/ \sqrt{2k t}).$$
\end{corollary}

On the other hand, Corollary \ref{exp-independence-a-cor} and Theorem \ref{heat-equation-a} imply:

\begin{theorem}\label{expectation-a}
  Let $y\in\mathbb R$, $k>0$, and  $(\tilde X_{t,k})_{t\ge0}$ a renormalized Bessel process starting in $x=(x_1,\ldots,x_N)\in C_N^A$.
  Then, for each 1-dimensional Brownian motion $(B_t)_{t\ge0} $ (starting in 0) independent from  $(\tilde X_{t,k})_{t\ge0}$,
  \begin{equation}\label{expect-final-new-a}
    \mathbb E\bigl(\prod_{i=1}^N (B_t+y- \tilde X_{t,k}^i)\bigr)=\prod_{i=1}^N (y- x_i).\end{equation}
\end{theorem}

\begin{proof} Corollary \ref{exp-independence-a-cor} with $Y=B_t+y$, Theorem \ref{heat-equation-a}, and Dynkin's formula yield
  $$\mathbb E\bigl(\prod_{i=1}^N (B_t+y- \tilde X_{t,k}^i)\bigr)=\mathbb E\bigl(\prod_{l=1}^N (B_t+y- x_l(t))\bigr) = \prod_{l=1}^N(y- x_l(0)).$$
\end{proof}

 Theorem \ref{heat-equation-a} has a variant for the stationary ODEs (\ref{ODE-tilde-a}).
 As the proof is completely analog to that of Theorem \ref{heat-equation-a}, we skip the proof.

\begin{theorem}\label{heat-equation-a-stat}
  Let $\psi:=(\psi_1,\ldots,\psi_N):[0,\infty[\to C_N^A$ be a differentiable function and $\lambda\ge0$ a constant.
      Then $\psi$ is a solution of
      \begin{equation}\label{ODE-tilde-a-gen}
\dot{\psi}_i(t) = \sum_{j: j\ne i}  \frac{1}{\psi_{i}(t)-\psi_{j}(t)}
-\lambda\psi_{i}(t)\quad(i=1,\ldots,N)
      \end{equation}
 in the sense of Theorem \ref{ode-ex-unique-a-thm}
      if and only if the function $H(t,z):=\prod_{i=1}^N (z-\psi_i(t))$
      solves the inverse `` stationary heat equation with potential''
      \begin{equation}\label{heateq-a-stationary}     H_t=  -\Bigl( \frac{1}{2}H_{zz} -\lambda z H_{z}\Bigr) -N\lambda\cdot H.
        \end{equation}
\end{theorem}

Please notice that for $\lambda=0$ this just gives Theorem \ref{heat-equation-a}.

\begin{remark}\label{martingale-a-stat}
\begin{enumerate}
\item[\rm{(1)}]  The proof of the if-part in Theorem \ref{heat-equation-a-stat}
  also works for functions $H$ which satisfy (\ref{heateq-a-stationary}) with quite arbitary potentials instead of
  $  -N\lambda$. This looks disturbing at a first glance and can be resolved  by the fact that by our general
  assumption, $H$ is a polynomial in $z$ of degree $N$ which is possible only for the potential $ -N\lambda$.
\item[\rm{(2)}]  For $\lambda>0$, the operator $L:=\frac{1}{2}\partial_{zz} -\lambda z \partial_{z}$
  is the generator of an Ornstein-Uhlenbeck process $(U_t)_{t\ge0}$, an asymptotically stationary variant of the
  Brownian motion. We conclude from Theorem \ref{heat-equation-a-stat} and
  the Feynman-Kac formula (see e.g. \cite{RW}) that for such a process and a function $H$ as described in the theorem,
  the process
  $(e^{N\lambda t} H(t, U_t))_{t\ge0}$ is a martingale. This fact can be used to derive an analogue of Theorem \ref{expectation-a} for 
 $(U_t)_{t\ge0}$ instead of a Brownian motion and an asyptotically stationary version of the Bessel processes $(\tilde X_t)_{t\ge0}$.
  However, as in these cases the same distributions appear up to  notations, the expectations in this case will be equivalent to those in
  (\ref{expect-final-new-a}).
\end{enumerate} 
\end{remark}

 \section{The Laguerre case}

In this section we study the ODE
\begin{equation}\label{ODE-b}
  x_i^\prime(t)= \sum_{j: j\ne i}\Bigl( \frac{1}{x_i-x_j}+  
 \frac{1}{x_i+x_j}\Bigr)+\frac{\nu}{x_i} = \sum_{j: j\ne i}\frac{2x_i}{x_i^2-x_j^2}+\frac{\nu}{x_i}
 \quad\quad (i=1,\ldots,N)
\end{equation}
for $N\ge1$ and some fixed parameter $\nu>0$ on  the Weyl chamber
$$C_N^B:=\{x\in\mathbb R^N: \quad 0\le x_1\le \ldots\le x_N\}.$$
This ODE can be treated similar  to the ODE  (\ref{ODE-a}).
We first recall that (\ref{ODE-b}) can be also uniquely solved for starting points in  $\partial C_N^B$
by \cite{VW3}:

\begin{theorem}\label{ode-ex-unique-b-thm}
For each $x(0)\in C_N^B$,  (\ref{ODE-b})
has a unique solution for  $t\ge0$ in  $ C_N^B$, i.e.,
 there is a unique continuous function
$x:[0,\infty[\to C_N^B$ with  $x(t)$ in the interior of  $ C_N^B$ 
with  (\ref{ODE-b}) for $t\in]0,\infty[$.
\end{theorem}

The ODEs
(\ref{ODE-b}) have  particular simple solutions which can be expressed in terms of the  zeros of the Laguerre polynomials
 $L_N^{(\nu-1)}$ where, for $\alpha>-1$, the  Laguerre polynomials $(L_n^{(\alpha)})_{n\ge 0}$  are orthogonal w.r.t.
 the density  $\mathrm{e}^{-x}x^{\alpha}$ on $]0,\infty[$; see  \cite{S}.
 
 A classical result of Stieltjes gives the following characterization of the zeros of $L_N^{(\nu-1)}$; see Section 6.7 of \cite{S} and also
  \cite{AKM2,AV1, V}:

\begin{lemma}\label{char-zero-B1}
Let $\nu>0$. For ${\bf y}=(y_1,\ldots,y_N)\in C_N^B$, the following statements  are equivalent:
\begin{enumerate}
\item[\rm{(1)}]  For $i=1,\ldots,N$, 
$$y_i=\sum_{j: j\ne i} \frac{2y_i}{y_i^2-y_j^2} +\frac{\nu}{y_i}=\sum_{j: j\ne i} \Bigl(
  \frac{1}{y_i-y_j} +\frac{1}{y_i+y_j}\Bigr) +\frac{\nu}{y_i}  .$$
\item[\rm{(2)}]  If $z_1^{(\nu-1)}\le\ldots\le z_N^{(\nu-1)}$ are the ordered zeros of $L_N^{(\nu-1)}$, then 
$$(z_1^{(\nu-1)},\ldots, z_N^{(\nu-1)})= (y_1^2, \ldots, y_N^2).$$
\end{enumerate}
\end{lemma}

This lemma  implies the following result; see  \cite{AV1, VW1, VW2, VW3}:

\begin{lemma}\label{special-solution-b}
 $x(t):=\sqrt{2t}\cdot{\bf y}  $ with $t\ge0$ is a solution of (\ref{ODE-b}) in the sense of Theorem \ref{ode-ex-unique-b-thm}.
\end{lemma}

The following facts can be  easily verified; see \cite{VW3}.

\begin{lemma}\label{prp-b}
\begin{enumerate}
  \item[\rm{(1)}] If $x(t)$ is a solution of
  (\ref{ODE-b}) , then for  $c\in\mathbb R$,
  $x(t+c)$ also solves  (\ref{ODE-b}), and for $c>0$,  $\frac{1}{\sqrt c} x(ct)$ is also a
    solution of  (\ref{ODE-b}).
\item[\rm{(2)}] For each solution $x(t)$ of  (\ref{ODE-b}) with start in  $ C_N^B$,
\begin{equation}\label{ODE-norm-b}
\|x(t)\|^2 =2N(N+\nu-1)t+\|x(0)\|^2.
\end{equation}
\end{enumerate}
\end{lemma}

 Lemma \ref{special-solution-b} and (\ref{ODE-norm-b})  imply the following known fact; see  e.g. \cite{AKM2}:
\begin{equation}\label{ODE-norm-laguerre-zeros}
  z_1^{(\nu-1)}+\ldots+z_N^{(\nu-1)}=N(N+\nu-1).
\end{equation}

We shall prove that for $t\to\infty$, all solutions of  (\ref{ODE-b}) tend to the particular solution 
in Lemma \ref{special-solution-b} up to a time shift as described  in Lemma \ref{prp-b}(1).
Before doing so, we first notice:

 \begin{lemma}\label{decreasing-error-b}
For all solutions $x(t),\tilde x(t)$ of  (\ref{ODE-b}),  $\|x(t)-\tilde x(t)\|$ is decreasing.
 \end{lemma}

 \begin{proof} We proceed as in the proof of Lemma \ref{decreasing-error-a}.
   Let $r_i(t):=x_i(t)-\tilde x_i(t)$ ($i=1,\ldots,N$) and $D(t):=\frac{1}{2}\sum_{i=1}^N r_i(t)^2$. Then
   \begin{align}\label{error-est-b}
D^\prime(t) &= \sum_{i=1}^N\Biggl( x_i^\prime(t)x_i(t) +\tilde x_i^\prime(t)\tilde x_i(t)- {\tilde x_i}^\prime(t)x_i(t)- x_i^\prime(t)\tilde x_i(t)
     \Biggr)\notag\\
     &= \sum_{i,j: i\ne j}\Biggl(\frac{x_i(t)}{x_i(t)-x_j(t)}+\frac{\tilde x_i(t)}{\tilde x_i(t)-\tilde x_j(t)}-
     \frac{ x_i(t)}{\tilde x_i(t)-\tilde x_j(t)}-\frac{\tilde x_i(t)}{ x_i(t)- x_j(t)} \Biggr)\notag\\
&\quad + \sum_{i,j: i\ne j}\Biggl(\frac{x_i(t)}{x_i(t)+x_j(t)}+\frac{\tilde x_i(t)}{\tilde x_i(t)+\tilde x_j(t)}-
     \frac{ x_i(t)}{\tilde x_i(t)+\tilde x_j(t)}-\frac{\tilde x_i(t)}{ x_i(t)+x_j(t)} \Biggr)\notag\\
     &\quad +\nu\sum_{i=1}^N \Biggl(2-\frac{\tilde x_i(t)}{ x_i(t)}-\frac{ x_i(t)}{\tilde x_i(t)}\Biggr)\notag\\
     &\le
     \sum_{1\le i< j\le N} \Biggl(2-  \frac{ x_j(t) - x_i(t)}{\tilde x_j(t)-\tilde x_i(t)}- \frac{\tilde x_j(t) -\tilde x_i(t)}{x_j(t)- x_i(t)}\Biggr)
\notag\\
&\quad +\sum_{1\le i< j\le N}
\Biggl(2-  \frac{ x_j(t) + x_i(t)}{\tilde x_j(t)+\tilde x_i(t)}- \frac{\tilde x_j(t) +\tilde x_i(t)}{x_j(t)+ x_i(t)}\Biggr)
\quad \le0
   \end{align}
 by the ordering of the components, and by $a+\frac{1}{a}-2\ge 0$ for  $a>0$.
   \end{proof}

 Please notice that different from Section 2, (\ref{error-est-b}) even implies that the error is strictly decreasing except for
 the trivial case $x(t)=\tilde x(t)$.

We next consider the angular parts  of   solutions  $x(t)$ of (\ref{ODE-b}).
A direct computation yields the following; see \cite{VW3}:

\begin{lemma}\label{stabilily-lemma1-b}
  For each starting point $x(0)\in C_N^B $, consider the solution $x(t)$  of (\ref{ODE-b}) and its angular part
  $\phi(t):=x(t)/\|x(t)\|$.
  Then the time-transformed angular part
  \begin{equation}\label{time-transform-b}
    \psi(t) := \phi\Bigl( N(N+\nu-1)t^2+ \|x(0)\|^2t \Bigr) \quad(t\ge0)
    \end{equation}
satisfies
\begin{align}\label{ODE-tilde-b}
{\psi}_i^\prime(t) = &\sum_{j\ne i}  \frac{1}{\psi_{i}(t)-\psi_{j}(t)}+
\sum_{j\ne i}  \frac{1}{\psi_{i}(t)+\psi_{j}(t)}
 \notag\\
&\quad
+\frac{ \nu}{  \psi_{i}(t)} -   N(N+\nu-1)\cdot \psi_{i}(t)
\end{align}
for $i=1,\ldots,N$ with $\psi(0)=x(0)/\|x(0)\|$. 
\end{lemma}

It is possible to derive a connection between solutions of (\ref{ODE-b}) and  (\ref{ODE-tilde-b}) analog to Lemma
\ref{connection-odes2-a}. This leads to the following analogue of Theorem \ref{ode-ex-unique-a-thm-stat}:

\begin{theorem}\label{ode-ex-unique-b-thm-stat}
For each $\psi(0)\in C_N^b$,  (\ref{ODE-tilde-b})
has a unique solution for  $t\ge0$ in  $ C_N^A$ in the sense of Theorem \ref{ode-ex-unique-b-thm}.
\end{theorem}

We next prove the following analogue of Proposition
 \ref{decreasing-error-a-stat}:

\begin{proposition}\label{decreasing-error-b-stat}
  Let  $\psi(t),\tilde \psi(t)$ be solutions of (\ref{ODE-tilde-b}) with $\psi(0),\tilde \psi(0)\in C_N^A$.
Then for $t\ge0$,
  $$\Bigl\|\psi(t)-\tilde \psi(t)\Bigr\|\le  e^{- N(N+\nu-1) t}\Bigl\|\psi(0)-\tilde \psi(0)\Bigr\|.$$ 
 \end{proposition}

 \begin{proof}
   Let $r_i(t):=\psi_i(t)-\tilde \psi_i(t)$ ($i=1,\ldots,N$) and $D(t):=\frac{1}{2}\sum_{i=1}^N r_i(t)^2$. Then, as
   in (\ref{error-est-b}),
\begin{align}
  D^\prime(t) = &\sum_{1\le i< j\le N} \Biggl(2-  \frac{ \psi_j(t) - \psi_i(t)}{\tilde \psi_j(t)-\tilde \psi_i(t)}-
     \frac{\tilde \psi_j(t) -\tilde \psi_i(t)}{\psi_j(t)- \psi_i(t)}
     \Biggr)\notag\\
&\quad+\sum_{1\le i< j\le N} \Biggl(2-  \frac{ \psi_j(t) + \psi_i(t)}{\tilde \psi_j(t)+\tilde \psi_i(t)}-
     \frac{\tilde \psi_j(t) +\tilde \psi_i(t)}{\psi_j(t)-+\psi_i(t)} \Biggr)\notag\\
&\quad +\nu\sum_{i=1}^N \Biggl(2-\frac{\tilde x_i(t)}{ x_i(t)}-\frac{ x_i(t)}{\tilde x_i(t)}\Biggr)\notag\\
     &- N(N+\nu-1)  \sum_{1=1}^N (\psi_i(t)^2 +\tilde\psi_i(t)^2-2\psi_i(t)\tilde\psi_i(t))\notag\\
     \le& - N(N+\nu-1) \sum_{1=1}^N (\psi_i(t)-\tilde\psi_i(t))^2 = -2N(N+\nu-1) D(t).
\notag
   \end{align}
The lemma of Gronwall implies the claim.
 \end{proof}

 The  proofs of Proposition \ref{decreasing-error-b-stat} mentioned in Remark \ref{other-proofs}
 for the Hermite case are also
 available  in the Laguerre case where again the eigenvalues of the associated Jacobi matrix can be determined via \cite{AV2}.
 We skip the details.

 Proposition \ref{decreasing-error-b-stat} and 
Lemma \ref{stabilily-lemma1-b} lead to the following convergence:

\begin{theorem}\label{final-convergence-b}
  Let  $x(t),\tilde x(t)$ be solutions of  (\ref{ODE-b}) on $C_N^B$ with
  $\|x(0)\|=\|\tilde x(0)\|$ in the sense of Theorem \ref{ode-ex-unique-b-thm}.
Then for all $t\ge0$,
$$\|x(t)-\tilde x(t)\|\le
\|x(0)-\tilde x(0)\|\cdot\sqrt{\frac{2N(N+\nu-1)}{\|x(0)\|^2}t +1}\cdot  e^{- \frac{1}{2} \Bigl(\sqrt{ 4N(N+\nu-1)t +\|x(0)\|^4}- \|x(0)\|^2\Bigr) }.$$

In particular, if  $x(t)$ is any solution   of   (\ref{ODE-a}) on $C_N^B$, then
\begin{align}
  \Bigl\|x(t)-\sqrt{2t+\frac{2\|x(0)\|^2}{N(N-1)}}\cdot & {\bf z }\Bigr\|\le\Biggl\|x(0)-\sqrt{\frac{2\|x(0)\|^2}{N(N-1)}}\cdot  {\bf z }
\Biggr\|\cdot\\
  &\cdot
\sqrt{\frac{2N(N+\nu-1)}{\|x(0)\|^2}t +1}\cdot  e^{- \frac{1}{2} \Bigl(\sqrt{ 4N(N+\nu-1)t +\|x(0)\|^4}- \|x(0)\|^2\Bigr) }
\notag\end{align}
\end{theorem}

\begin{proof}  Be Theorem \ref{ode-ex-unique-b-thm}, the result is trivial for $x(0)=\tilde x(0)=0$. We now assume that
   $\|x(0)\|=\|\tilde x(0)\|>0$.
Eq.~ (\ref{time-transform-b}) motivates the time transform
  $$t= N(N+\nu-1)\tau^2+ \|x(0)\|^2\tau$$
  with $t,\tau\ge 0$ and
  $$\tau = \frac{1}{2N(N+\nu-1)}\Bigl(\sqrt{ 4N(N+\nu-1)t +\|x(0)\|^4}- \|x(0)\|^2\Bigr).$$
   Proposition \ref{decreasing-error-b-stat}, Lemma \ref{stabilily-lemma1-b}, and Lemma \ref{prp-b}(2) imply that
  \begin{align}
    \|x(t)&-\tilde x(t)\|= \|x(t)\|\cdot
    \Bigl\| \frac{x(t)}{\|x(t)\|}- \frac{\tilde x(t)}{\|\tilde x(t)\|}\Bigr\|  \notag\\
  &\le\|x(t)\|\cdot \Bigl\| \frac{x(0)}{\|x(0)\|}- \frac{\tilde x(0)}{\|\tilde x(0)\|}\Bigr\|
 e^{- N(N+\nu-1)  \tau}\notag\\
&= \sqrt{\frac{2N(N+\nu-1)}{\|x(0)\|^2}t +1}\cdot \|x(0)-\tilde x(0)\|\cdot  e^{- \frac{1}{2} \Bigl(\sqrt{ 4N(N+\nu-1)t +\|x(0)\|^4}- \|x(0)\|^2\Bigr) }.
 \notag \end{align}
This proves the first estimate. The second estimate is then also clear. 
\end{proof}

We next discuss a connection  of the ODE (\ref{ODE-b}) to the
inverse heat equations related to one-dimensional Bessel processes. For these processes we refer e.g.~to Section XI.1 of
 \cite{RY}.
 We recapitulate that a Bessel process $(Y_{t,\alpha})_{t\ge0}$ on $[0,\infty[$ of index $\alpha>-1$ is a Feller diffusion with
     reflecting boundary at 0 satisfying the SDE
     $$dY_{t,\alpha}=dB_t + \frac{\alpha+1/2}{Y_{t,\alpha}}dt$$
     with some one-dimensional Brownian motion $(B_t)_{t\ge0}$,
     i.e., $(Y_{t,\alpha})_{t\ge0}$ has the generator
     $$G_\alpha f(x):=\frac{1}{2}f^{\prime\prime}(x)+ \frac{\alpha+1/2}{x}f^{\prime}(x)$$
       for even $\mathcal C^2$-functions $f$ on $\mathbb R$.
This evenness  fits to the squares in Lemma \ref{char-zero-B1}(2).

\begin{theorem}\label{heat-equation-b}
  Let $\nu>0$.
  Let $x:=(x_1,\ldots,x_N):[0,\infty[\to C_N^B$ be.
      Then $x$ is a solution of (\ref{ODE-b}) in the sense of Theorem \ref{ode-ex-unique-b-thm}
      if and only if the function
      $$H(t,z):=\prod_{i=1}^N (z^2-x_i(t)^2) \quad\quad(z\in\mathbb R,\> t\ge0)$$
      solves the inverse Bessel-heat equation
      \begin{equation}\label{Bessel-heat}
        \partial_tH = -\Bigl(\frac{1}{2}\partial_{zz}+ \frac{\nu-1/2}{z}\partial_z\Bigr)H=-G_{\nu-1}H
      \end{equation}
      where the operator  $G_{\nu-1}$ acts w.r.t.~the variable $z$.
\end{theorem}

\begin{proof} We proceed as in the proof of Theorem \ref{heat-equation-a}. 
 Assume first that $x(t)$ satisfies  (\ref{ODE-b}), and consider  $H$ as defined in the theorem. For $i,j=1,\ldots,N$ with $i\ne j$
 consider the polynomials $H_i$ and $H_{i,j}$ in $z$ with
 $H_i(t,z):=H(t,z)/(z^2-x_i(t)^2)$ and  $H_{i,j}(t,z):=H(t,z)/((z^2-x_i(t)^2)(z^2-x_j(t)^2))$.
  Then,  for $i\ne j$,
  \begin{equation}\label{two-to-one-b}
\frac{x_i(t)^2 \cdot H_i(t,z)-x_j(t)^2 \cdot H_j(t,z)}{x_i(t)^2-x_j(t)^2}= z^2\cdot H_{i,j}(t,z).
    \end{equation}
  Hence, by (\ref{ODE-b}),
  \begin{align}\label{comp-b1-heat}
    \partial_t H(t,z)=&    -2\sum_{i=1}^N x_i^\prime(t) x_i(t) \> H_i(t,z)\\
    =&-4\sum_{i,j: i\ne j} \frac{x_i(t)^2\cdot H_i(t,z)}{x_i(t)^2-x_j(t)^2}  - 2\nu\sum_{i=1}^N  H_i(t,z)
   \notag\\
   =& - 2  \sum_{i,j: i\ne j }z^2\cdot H_{i,j}(t,z)-  2\nu\sum_{i=1}^N  H_i(t,z).
\notag
  \end{align}
  Moreover,
   \begin{align}\label{comp-b2-heat}
     \partial_z H(t,z)=& 2z \sum_{i=1}^NH_i(t,z),\\
       \partial_{zz} H(t,z)=& 2 \sum_{i=1}^NH_i(t,z)+ 4z^2  \sum_{i,j: i\ne j } H_{i,j}(t,z).
 \notag\end{align}
 (\ref{comp-b1-heat}) and (\ref{comp-b2-heat}) now lead to (\ref{Bessel-heat}) as claimed.

  Now assume that (\ref{Bessel-heat}) holds. Using the  $H_i(t,z)$ above, we have
  \begin{align}
\partial_{z} H(t,z)&= 2z\cdot  H_i(t,z)+ \bigl( z^2-x_i(t)^2)\partial_{z} H_i(t,z),\notag\\
\partial_{zz} H(t,z)&= 2 H_i(t,z)+ 4z\cdot \partial_{z} H_i(t,z)+ (z^2-x_i(t)^2)\partial_{zz} H_i(t,z),
\notag\end{align}
and thus
     \begin{equation}\label{log-derivative-b2}
       \frac{\partial_{zz}H(t,x_i(t))}{\partial_{z}H(t,x_i(t))}= 2\frac{\partial_{z} H_i(t,x_i(t))}{H_i(t,x_i(t))}
+\frac{1}{x_i(t)}.
    \end{equation}
     This, (\ref{implicit}), and the Bessel heat equation (\ref{Bessel-heat}) now lead to
 \begin{align}
   x_i^\prime(t)&=- \frac{\partial_{t} H(t,x_i(t))}{\partial_{z} H(t,x_i(t))} \\
&= \frac{ \frac{1}{2} \partial_{zz}H(t,x_i(t))+ \frac{\nu-1/2}{x_i(t)}\partial_zH(t,x_i(t))}{\partial_zH(t,x_i(t))}\notag\\
&= 
 \frac{\nu-1/2}{x_i(t)}+ \frac{1}{2}\frac{  \partial_{zz}H(t,x_i(t))}{  \partial_{z}H(t,x_i(t))}\notag\\
&=  \frac{\nu}{x_i(t)}+\frac{  \partial_{z}H_i(t,x_i(t))}{ H_i(t,x_i(t))} =
 \frac{\nu}{x_i(t)}+\sum_{j:j\ne i} \frac{2x_i(t)}{x_i(t)^2-x_j(t)^2}
   \notag\end{align}
     as claimed.
\end{proof}

\begin{remark} As in the preceding section, Dynkin's formula ensures that for polynomial  solutions $H$
  of the inverse  Bessel-heat equation equation $H_t+ G_{\nu-1}H=0$ and for associated Bessel processes 
 $(Y_{t,\nu-1})_{t\ge0}$ on $[0,\infty[$, the process  $(H(t,Y_{t,\nu-1} ))_{t\ge0}$ is a martingale.

      By Lemma \ref{special-solution-b} and Theorem \ref{heat-equation-b},
examples of such  Bessel-type space-time-harmonic functions are the so called Bessel-type heat polynomials
\begin{equation}\label{heat-pol-b}
   H_N(t,z):=\prod_{i=1}^N (z^2-2t\cdot z_i^{(\nu-1)})=
   c_N\cdot t^{N}\cdot L_N^{(\nu-1)}(z^2/2t)
   \end{equation}
with some constants $c_N$. This connection between Bessel processes and Laguerre polynomials is  well known.
\end{remark}

Similar to Theorem \ref{expectation-a} we now 
 to derive some expectations for Bessel processes of type B on  $C_N^B$
as studied in \cite{CGY, RV1, AKM2, AV1, AV2, V3}. In view of Theorem \ref{expectation-b} below, we use parameters
$\nu,\beta>0$ and define the Bessel processes  $(X_{t,\nu,\beta}=(X_{t,\nu,\beta}^1,\ldots,X_{t,\nu,\beta}^N))_{t\ge0}$
as the unique strong solutions of the SDEs
 \begin{equation}\label{SDE-B}
dX_{t,\nu,\beta}^i = d\tilde B_{t,i}+ \beta \sum_{j: j\ne i}
\Bigl(\frac{1}{X_{t,\nu,\beta}^i-X_{t,\nu,\beta}^j}+
\frac{1}{X_{t,\nu,\beta}^i+X_{t,\nu,\beta}^j}\Bigr)dt
+ \frac{\nu\cdot\beta}{X_{t,\nu,\beta}^i}dt \end{equation}
 for $i=1,\ldots,N$ with reflecting boundaries
 with some $N$-dimensional Brownian motion $(\tilde B_t)_{t\ge0}$ for arbitrary starting points $x\in C_N^B$.
  $(X_{t,\nu,\beta})_{t\ge0}$ is a Feller diffusion with generator
\begin{equation}\label{generator-laguerre-intro}
  L_{\nu,\beta}f(x):= \frac{1}{2} \Delta f(x) + \sum_{i=1}^N\Biggl( 
  \beta \sum_{j: j\ne i} \Bigl(\frac{1}{x_i-x_j}+\frac{1}{x_i+x_j}\Bigr)+
\frac{\nu\beta}{x_i}\Biggr) \frac{\partial}{\partial x_i}f(x) 
 \end{equation}
 for
 \begin{align}
   f\in D(L_k):=\{f\in \cal C^{2}(\mathbb R^N), 
   \>\>\> f\>\>&\text{ invariant under all permutations of coordinates}\notag\\
&\text{ and even in all coordinates}\}
 \notag \end{align} 
 where again the transition probabilities are known in terms of multivariate Bessel functions; see \cite{R1, R2,  RV1, V3}.
The  renormalized processes $(\tilde X_{t,\nu,\beta}:=X_{t/\beta,\nu,\beta})_{t\ge0}$ then are solutions of
\begin{equation}\label{SDE-B-normalized}
d\tilde X_{t,\nu,\beta}^i = \frac{1}{\sqrt \beta}d\tilde B_{t,i}+  \sum_{j: j\ne i}
\Bigl(\frac{1}{\tilde X_{t,\nu,\beta}^i-\tilde X_{t,\nu,\beta}^j}+
\frac{1}{\tilde X_{t,\nu,\beta}^i+\tilde X_{t,\nu,\beta}^j}\Bigr)dt
+ \frac{\nu}{\tilde X_{t,\nu,\beta}^i}dt \end{equation}
for $i=1,\ldots,N$, and  for  $\beta=\infty$, this   degenerates into the ODE (\ref{ODE-b}).

For these renormalized processe and the elementary
symmetric polynomials $e_l^N$ we have by Corollary 3.5 in \cite{KVW}:

\begin{lemma}\label{indepenndence-exp-b}
For  $x\in\mathbb R^N$ let $x^2\in\mathbb R^N $ be the vector where all coordinates are squared.
Then for  $t\ge0$, $l=0,\ldots,N$, $\beta\in]0,\infty]$, and each fixed starting point $x\in C_N^B$,
 the expectations
$\mathbb E( e_l^N(\tilde X_{t,\nu,\beta}^2) )$ depend only on $\nu+1/(2\beta)$.
Hence,
 $$E( e_l^N(\tilde X_{t,\nu,\beta}^2))=  e_l^N(x(t,\nu+1/2\beta)^2)$$
 for the solution $x(t,\nu+1/2\beta)$ of  (\ref{ODE-b}) with start in
 $x(0)=x$ and with parameter $\nu+1/(2\beta)$ instead of $\nu$.
\end{lemma}

This yields:

\begin{corollary}\label{exp-independence-b-cor} 
  Assume the notations of Lemma \ref{indepenndence-exp-b}.
  Then,  for
 $t\ge0$ and each $\mathbb R$-valued random variable $Y$ with $N$-th moment, which is independent from $\tilde X_{t,\nu,\beta}$,
$$\mathbb E\bigl(\prod_{i=1}^N (Y- (\tilde X_{t,\nu,\beta}^i)^2)\bigr)=\mathbb E\bigl(\prod_{i=1}^N (Y- x_i(t,\nu+1/2\beta)^2)\bigr).$$
\end{corollary}

\begin{proof}
\begin{align}\label{det-computation-1a}
&\mathbb E\bigl(\prod_{i=1}^N (Y- (\tilde X_{t,\nu,\beta}^i)^2)\bigr)=
\sum_{l=0}^N (-1)^l\mathbb E\bigl(e_l^N( \tilde X_{t,\nu,\beta}^2 ) Y^{N-l}\bigr)\notag\\
&=\sum_{l=0}^N (-1)^l\mathbb E(e_l^N( \tilde X_{t,\nu,\beta}^2 )) \cdot\mathbb E(Y^{N-l}
=\sum_{l=0}^N (-1)^l  e_l^N(x(t,\nu+1/2\beta)^2) \cdot \mathbb E(Y^{N-l})   \notag\\
&=\mathbb E\bigl(\prod_{i=1}^N (Y- x_i(t,\nu+1/2\beta)^2)\bigr) .
\notag\end{align}
\end{proof}

For $Y=y\in\mathbb R$ a constant and the starting point $x=0\in C_N^B$, this and the definition of
the Laguerre polynomials $L_N^{(\alpha)}$ yield the following; see \cite{KVW}:

\begin{corollary}\label{exp-independence-b-cor1}
Let $( X_{t,\nu,\beta})_{t\ge0}$ be the Bessel process of type B
starting in 0 with parameters  $\nu\ge 0,\beta>0$.
Then, 
$$\mathbb E\bigl(\prod_{i=1}^N (y- (X_{t,\nu,\beta}^i)^2)\bigr)  = (-1)^N(2t \beta)^{N}\cdot N!\cdot L_N^{(\nu+1/(2\beta)-1)}(y/(2t\beta)).$$
\end{corollary}

On the other hand, Corollary \ref{exp-independence-b-cor}  and Theorem \ref{heat-equation-b} imply:

\begin{theorem}\label{expectation-b}
  Let $y\in [0,\infty[$, $\nu,\beta>0$, and  $(\tilde X_{t,\nu,\beta})_{t\ge0}$ a renormalized Bessel process of type B
  starting in $x=(x_1,\ldots,x_N)\in C_N^B$. Moreover, let $(Y_t)_{t\ge0} $
  be a  one-dimensional Bessel process  with index $\alpha= \nu+1/(2\beta)   -1>-1$ starting in
  $y$ such that this process is
independent from  $(\tilde X_{t,\nu,\beta})_{t\ge0}$. Then,
\begin{equation}\label{equ-expectation-b}
\mathbb E\bigl(\prod_{i=1}^N (Y_t^2- (\tilde X_{t,\nu,\beta}^i)^2)\bigr)=\prod_{i=1}^N (y^2- x_i^2).
\end{equation}
\end{theorem}

\begin{remark}
   Theorems \ref{expectation-a} and \ref{expectation-b} hold also for Dunkl processes on $\mathbb R^N$ of types A and B
  (see \cite{CGY, R1, R2, RV1, V3} for details) instead of Bessel processes as the additional jumps there  have no
  influence to the expectations there.

  Moreover, many results in Sections 2 and 3 can be derived
  for general multivariate Bessel processes associated with  root systems. However, if one compares 
 Sections 2 and 3, it turns out that still many details  depend heavily on the concrete case.
\end{remark}

We next consider a stationary variant of Theorem \ref{heat-equation-b}
for the stationary ODEs (\ref{ODE-tilde-b}) which corresponds to Theorem \ref{heat-equation-b-stat}.
 As the proof is completely analog to that of Theorem \ref{heat-equation-b}, we  skip the proof.

\begin{theorem}\label{heat-equation-b-stat}
  Let $\psi:=(\psi_1,\ldots,\psi_N):[0,\infty[\to C_N^B$ be a differentiable function and $\lambda\ge0$, $\nu>0$  constants.
      Then $\psi$ is a solution of
      \begin{equation}\label{ODE-tilde-b-gen}
\psi^\prime_i(t) = \sum_{j: j\ne i} \Bigl( \frac{1}{\psi_{i}(t)-\psi_{j}(t)}+ \frac{1}{\psi_{i}(t)+\psi_{j}(t)}\Bigr)+\frac{\nu}{\psi_i}
-\lambda\psi_{i}(t)
      \end{equation}
for $i=1,\ldots,N$ in the sense of Theorem \ref{ode-ex-unique-b-thm}
      if and only if the function $H(t,z):=\prod_{i=1}^N (z^2-\psi_i(t)^2)$
      solves the inverse ``stationary heat equation with potential''
      $$H_t=  -\Bigl( \frac{1}{2}H_{zz} + \frac{\nu-1/2}{z}H_z-\lambda z H_{z}\Bigr) -N\lambda\cdot H. $$
\end{theorem}

Please notice that for $\lambda=0$ this just gives Theorem \ref{heat-equation-b}.

\section{The Jacobi case}

In this section we study the ODEs
\begin{equation}\label{ODE-jacobi}
\frac{d}{dt}x_i(t)
		=(p-q)-(p+q)x_i(t)
			+2\sum_{j:\> j\neq i}\frac{1-x_i(t)
			x_j(t)}{x_i(t)-x_j(t)}\quad ( i=1,\dots,N)
\end{equation}
on the alcoves $A_N:=\{x\in \mathbb R^N: \> -1\le x_1\le \ldots\le x_N\le 1\}$
for parameters $p,q>N-1$. These ODEs  appear as freezing limits of the SDEs of multivariate Jacobi processes on  $A_N$;
see \cite{De, RR, AVW, RV2, V2, V3} for these processes, and to \cite{HO, HeS}
for the  related harmonic analysis and special functions. The ODE (\ref{ODE-jacobi}) can be also written as
\begin{equation}\label{ODE-jacobi-mod}
  \frac{d}{dt}x_i(t)
		=(p-q)-(p+q-2N+2)x_i(t)
			+2\sum_{j:\> j\neq i}\frac{1-x_i(t)^2}{x_i(t)-x_j(t)}\quad ( i=1,\dots,N).
\end{equation}

The ODEs (\ref{ODE-jacobi})  are closely related to the zeros of the
  Jacobi polynomials
$(P_N^{(\alpha,\beta)})_{N\ge 0}$ on $[-1,1]$ with the parameters
$$\alpha:=q-N>-1, \quad \beta:=p-N>-1,$$
where the $P_N^{(\alpha,\beta)}$ are orthogonal 
 w.r.t.~the weights $(1-x)^\alpha(1+x)^\beta$ on $[-1,1]$; see Ch.~4 of  \cite{S}.
  
We  need the following  fact on the ordered zeros 
$-1<z_1< \ldots< z_N<1$ of  $P_N^{(\alpha,\beta)}$; see Theorem 6.7.1 of \cite{S} or \cite{HV}.

\begin{lemma}\label{char-zero-jacobi}
  Let  $N\in\mathbb{N}$ and $p,q> N-1$. Let $\alpha:=q-N>-1, \beta:=p-N>-1$. Then
 for $ {\bf z}=(z_1,\ldots,z_N)\in A_N$, the following  are equivalent:
\begin{enumerate}
\item[\rm{(1)}]$\prod_{i=1}^N((1-x_i)^{q+1-N} (1+x_i)^{p+1-N}) \cdot\prod_{i,j: \> i<j} (x_i-x_j)^2$
 is maximal in $ {\bf z}\in A_N$;
\item[\rm{(2)}] For $i=1,\ldots,N,$
  \begin{equation}
          \sum_{j: j\neq i}^{N}\frac{1}{z_i-z_j}+\frac{\alpha+1}{2}\frac{1}{z_i-1}+\frac{\beta+1}{2}\frac{1}{z_i+1}=0.
          	\end{equation}
\item[\rm{(3)}] $z_{1}<\ldots< z_{N}$ are
  the ordered zeros  of  $P_N^{(\alpha,\beta)}$.
\end{enumerate}
\end{lemma}

This lemma ensures that the vector $ {\bf z}$ in the lemma is the only stationary solution of the ODE (\ref{ODE-jacobi-mod}) and thus (\ref{ODE-jacobi})
in $A_N$.
Moreover, by Section 6 of \cite{AVW}, this solution attracts all solutions, and  the ODE even can start on the boundary:

  \begin{theorem}\label{main-solutions-ode-jacobi}
          Let $N\in\mathbb{N}$ and $p,q> N-1$. Then for each
		 each $x_0\in A_N$, \eqref{ODE-jacobi} has a unique 
		 solution $x(t)$ for  $t\geq0$ in the sense as in Theorem \ref{ode-ex-unique-a-thm}.
                  
		Moreover, for all solutions
		$\lim_{t\to\infty}x(t)= {\bf z}$.
  \end{theorem}

  The proof of the last statement of this lemma in  \cite{AVW} uses that \eqref{ODE-jacobi} can be interpreted as a gradient system. More precisely,
  the Heckman-Opdam theory in \cite{HO, HeS} motivates to transform the ODE by using
  $$x_i=:\cos \tau_i \quad\text{for}\quad i=1,\ldots,N$$
  with
  $$\tau\in\tilde A_N:=\{ \tau\in \mathbb R^N: \> 
  \pi\ge \tau_1\ge\ldots\ge \tau_N\ge0\}.$$
Elementary calculus (see e.g.~the computations in  the appendix of \cite{AVW}) then yields that  \eqref{ODE-jacobi} then has the following form as a gradient system:
   \begin{equation}\label{trig-ode}
        \begin{split}
        \frac{d}{dt} \tau_i(t)=&
           (q-p)\cot\left(\frac{\tau_i(t)}{2}\right)+2(p+1-N) \cot(\tau_i(t))\\
          &+\sum_{j: j\ne i}\left(\cot\left(\frac{\tau_i(t)-\tau_j(t)}{2}\right)+
          \cot\left(\frac{\tau_i(t)+\tau_j(t)}{2}\right)\right).
          \end{split}
          \end{equation}

We next  estimate the order  of convergence for the the stationary solutions.
Unfortunately we were not able to transfer the arguments of Propositions 
\ref{decreasing-error-a-stat} and  \ref{decreasing-error-b-stat} to  \eqref{ODE-jacobi} directly.
However,  this works in the  transformed trigonometric coordinates.
This effect, that some computations are easier in trigonometric coordinates than in algebraic ones,
appears also in \cite{HV} for the spectra of covariance matrices in some freezing central limit theorem for Jacobi ensembles.

 \begin{proposition}\label{decreasing-error-jacobi-stat-trigono}
  Let  $\tau(t),\tilde \tau(t)$ be solutions of (\ref{trig-ode}) with $\tau(0),\tilde \tau(0)\in \tilde A_N$.
Then for $t\ge0$,
$$\Bigl\|\tau(t)-\tilde \tau(t)\Bigr\|\le  e^{- c  t}\Bigl\|\tau(0)-\tilde \tau(0)\Bigr\|  \quad\text{with}\quad
c=\frac{p+q+2\cdot min(p,q)+2-2N}{4}>\frac{N-1}{2}\ge0.$$ 
  \end{proposition}

 \begin{proof} We proceed as for  Propositions \ref{decreasing-error-a-stat} and \ref{decreasing-error-b-stat}.
   Let $r_j(t):=\tau_j(t)-\tilde \tau_j(t)$ for $j=1,\ldots,N$, and $D(t):=\frac{1}{2}\sum_{j=1}^N r_j(t)^2$.
   Then,  by (\ref{trig-ode}),
\begin{equation}
  D^\prime(t) = 
  \sum_{j=1}^N\Bigl( \tau_j^\prime(t)- \tilde \tau_j^\prime(t)\Bigr)\Bigl( \tau_j(t)- \tilde \tau_j(t)\Bigr)
=:  A_1(t)+A_2(t) \end{equation}
with
\begin{align}
A_1(t)= \sum_{l,j: \> l\ne j}\Bigl( \tau_j(t)- \tilde \tau_j(t)\Bigr)\Bigl(&\cot\Bigl(\frac{\tau_j(t)-\tau_l(t)}{2}\Bigr)-
\cot\Bigl(\frac{\tilde \tau_j(t)-\tilde \tau_l(t)}{2}\Bigr)\Bigr)\notag\\
&+\Bigl(\cot\Bigl(\frac{\tau_j(t)+\tau_l(t)}{2}\Bigr)-
\cot\Bigl(\frac{\tilde \tau_j(t)+\tilde \tau_l(t)}{2}\Bigr)\Bigr)\notag
\end{align}
and
\begin{align}
  A_2(t)=\sum_{i=1}^N \Bigl( \tau_j(t)- \tilde \tau_j(t)\Bigr)
  \cdot\Bigl( &(q-p)\Bigl(\cot\left(\frac{\tau_i(t)}{2}\right)-\cot\left(\frac{\tilde\tau_i(t)}{2}\right)\Bigr)\notag\\
&+2(p+1-N)(\cot\tau_i(t) -\cot\tilde\tau_i(t))\Bigr).
\notag
\end{align}
In order to estimate  $A_2(t)$ we conclude from the mean value theorem that
\begin{align}\label{sine-estimate}
(q-p)&\Bigl(\cot\Bigl(\frac{\tau_i(t)}{2}\Bigr)-\cot\Bigl(\frac{\tilde\tau_i(t)}{2}\Bigr)\Bigr)
+2(p+1-N)(\cot\tau_i(t) -\cot\tilde\tau_i(t))\notag\\
&= -(\tau_i(t)-\tilde\tau_i(t))\Bigl(\frac{q-p}{2\sin^2 (\hat\tau_i(t)/2)} + \frac{2(p+1-N)}{\sin^2 (\hat\tau_i(t))}\Bigr)
\end{align}
with some $\hat\tau_i(t)$ between $\tau_i(t)$ and  $\tilde\tau_i(t)$. This implies for $q\ge p$ that
\begin{equation}\label{est-jacobi-a2}  A_2(t)\le -\frac{q+3p+4-4N}{2}\sum_{i=1}^N \Bigl( \tau_j(t)- \tilde \tau_j(t)\Bigr)^2.
\end{equation}
In order to handle  $A_1(t)$, we write it as
\begin{align}
 & A_1(t)= \notag\\ =& 
  \sum_{l,j: \> l> j}\Bigl( \tau_j(t)+ \tau_l(t) -\tilde \tau_j(t)-\tilde \tau_l(t)\Bigr)
 \Bigl(\cot\Bigl(\frac{\tau_j(t)+\tau_l(t)}{2}\Bigr)-
\cot\Bigl(\frac{\tilde \tau_j(t)+\tilde \tau_l(t)}{2}\Bigr)\Bigr)\notag\\
+&
  \sum_{l,j: \> l> j}\Bigl( \tau_j(t)- \tau_l(t) -\tilde \tau_j(t)+\tilde \tau_l(t)\Bigr)
 \Bigl(\cot\Bigl(\frac{\tau_j(t)-\tau_l(t)}{2}\Bigr)-
\cot\Bigl(\frac{\tilde \tau_j(t)-\tilde \tau_l(t)}{2}\Bigr)\Bigr).\notag
\end{align}
Hence, again by the mean value theorem for the cotangens, for some $\tau_{j,l,1},\tau_{j,l,2}\in [0,\pi]$,
\begin{align}
  A_1(t)= & -\frac{1}{2} \sum_{l,j: \> l>j}\Bigl( \tau_j(t)+ \tau_l(t) -\tilde \tau_j(t)-\tilde \tau_l(t)\Bigr)^2\frac{1}{\sin^2( \tau_{j,l,1})}
\notag\\
&-\frac{1}{2} \sum_{l,j: \> l>j}\Bigl( \tau_j(t)- \tau_l(t) -\tilde \tau_j(t)+\tilde \tau_l(t)\Bigr)^2\frac{1}{\sin^2( \tau_{j,l,2})}\notag\\
\le  -\frac{1}{4} &\sum_{l,j: \> l\ne j}\Bigl(\Bigl( \tau_j(t)+ \tau_l(t) -\tilde \tau_j(t)-\tilde \tau_l(t)\Bigr)^2
+\Bigl( \tau_j(t)- \tau_l(t) -\tilde \tau_j(t)+\tilde \tau_l(t)\Bigr)^2\Bigr).\notag
\end{align}
As
\begin{align}
  \Bigl( \tau_j(t)+ &\tau_l(t) -\tilde \tau_j(t)-\tilde \tau_l(t)\Bigr)^2+
  \Bigl( \tau_j(t)-\tau_l(t) -\tilde \tau_j(t)+\tilde \tau_l(t)\Bigr)^2\notag\\
&= 2(\tau_j(t)-\tilde \tau_j(t))^2 +2(\tau_l(t)-\tilde \tau_l(t))^2,\end{align}
we conclude that
\begin{equation}\label{est-jacobi-a1}
  A_1(t)\le -(N-1)\sum_{j=1}^N (\tau_j(t)-\tilde \tau_j(t))^2.\end{equation}
In summary we obtain from (\ref{est-jacobi-a1}) and (\ref{est-jacobi-a2}) for $q\ge p$ that
$$D^\prime(t) \le   -\frac{q+3p+2-2N}{2} D(t).$$
This and the lemma of Gronwall yield the claim for $q\ge p$.

For the case $q\le p$ we notice that the ODE (\ref{ODE-jacobi}) is invariant under the transform $x(t)\mapsto -x(t)$
when the parameters $p$ and $q$ are interchanged. This leads to the theorem for $q\le p$.
 \end{proof}

 The exponent $c$ in Proposition \ref{decreasing-error-jacobi-stat-trigono} is not optimal. In particular, in the step from 
 (\ref{sine-estimate}) to (\ref{est-jacobi-a2}) we loose some information. However, the computation of the optimal constant in
(\ref{est-jacobi-a2}) depending on $p,q,N$ seems to be nasty.
 
  We next consider the connection of the ODEs (\ref{ODE-jacobi}) to some inverse heat equations. As we here are in an
 asymptotic stationary case, we obtain a result which corresponds to the Ornstein-Uhlenbeck cases in Theorems
 \ref{heat-equation-a-stat} and \ref{heat-equation-b-stat}.

 \begin{theorem}\label{heat-equation-jacobi}
  Let $x:=(x_1,\ldots,x_N):[0,\infty[\to A_N$ be a differentiable function.
      Then  $x(t)$ is a solution of (\ref{ODE-jacobi})
      if and only if  $H(t,z):=\prod_{i=1}^N (z-x_i(t))$
      solves the inverse `` Jacobi-type heat equation with potential''
 \begin{equation}\label{heat-jacobi}
   H_t=  -\Biggl( (1-z^2) H_{zz}+ \Bigl( (p-q)+(2(N-1)-(p+q))z\Bigr) H_{z} \Biggr) -N(p+q-N+1)H.
   \end{equation}
 \end{theorem}

 \begin{proof}
   Assume first that $x(t)$ satisfies  (\ref{ODE-jacobi}).
   Consider $H$ as defined in the theorem, and, for $i\ne j$,
   \begin{equation}\label{def-h-jacobi}  H_i(t,z):=H(t,z)/(z-x_i(t)), \quad
     H_{i,j}(t,z):=H(t,z)/((z-x_i(t))(z-x_j(t))). \end{equation}
   Then (\ref{ODE-jacobi}),
   \begin{equation}\label{h1-jacobi} 
 \partial_z H(t,z)=\sum_{i=1}^N H_i(t,z), \quad\quad 
 \partial_{zz} H(t,z)=\sum_{i,j: \> i\ne j} H_{i,j}(t,z), \end{equation}
and (\ref{two-to-one})  imply that
\begin{align}\label{h2-jacobi}
  \partial_t H(t,z)=&    -\sum_{i=1}^N x_i^\prime(t) \> H_i(t,z)\\
     =& -(p-q)\sum_{i=1}^N H_i(t,z)+(p+q) \sum_{i=1}^N ((x_i(t)-z)+z)H_i(t,z)\notag\\
&\quad -2\sum_{i,j: \> i\ne j} \frac{1-x_i(t)x_j(t)}{x_i(t)-x_j(t)}\> H_i(t,z)\notag\\
 =& -(p-q)\partial_z H(t,z)+(p+q) \Bigl( z\cdot\partial_z H(t,z)-N\cdot H(t,z)\Bigr)\notag\\
&\quad -2\sum_{i,j: \> i\ne j}\frac{H_i(t,z)}{x_i(t)-x_j(t)} +2 \sum_{i,j: \> i\ne j}
\frac{H_i(t,z)x_i(t)x_j(t)}{x_i(t)-x_j(t)}\notag\\
 =& -(p-q)\partial_z H(t,z)+(p+q) z\cdot\partial_z H(t,z)- N(p+q)H(t,z)\notag\\
&\quad  -\sum_{i,j: \> i\ne j}  H_{i,j}(t,z)+ \sum_{i,j: \> i\ne j} H_{i,j}(t,z)x_i(t)x_j(t)\notag
\end{align}
where in the last $=$ in the last two sums, two summands of the LHS correspond to one on the RHS.
The second last sum in the last formula of (\ref{h2-jacobi}) can be treated via (\ref{h1-jacobi}).  For the last sum we also
observe from  (\ref{h1-jacobi}) that
\begin{align}\label{h3-jacobi}
\sum_{i,j: \> i\ne j} &H_{i,j}(t,z)x_i(t)x_j(t)\\
=&\sum_{i,j: \> i\ne j}\Bigl( (z-x_i(t))(z-x_j(t))+z^2 +z((x_i(t)-z)+(x_j(t)-z))\Bigr)  H_{i,j}(t,z)\notag\\
=& N(N-1) H(t,z)+ z^2 \partial_{zz}H(t,z)-2(N-1)z\partial_z H(t,z).
\notag\end{align}
  (\ref{h2-jacobi}) and (\ref{h3-jacobi}) now lead to the inverse Jacobi-type heat equation
(\ref{heat-jacobi}) as claimed.

Now assume that (\ref{heat-jacobi}) holds. As in the proof of Theorem \ref{heat-equation-a}, we obtain from
(\ref{implicit}) and (\ref{heat-jacobi}) that for $i=1,\ldots,N$,
  \begin{align}\label{h4-jacobi}
    x_i^\prime(t)&=- \frac{\partial_{t} H(t,x_i(t))}{\partial_{z} H(t,x_i(t))} \\&=
    \frac{ (1-x_i(t)^2)\partial_{zz} H(t,x_i(t)) + \Bigl( (p-q)+(2(N-1)-(p+q))x_i(t)\Bigr)\partial_{z} H(t,x_i(t))}{\partial_{z} H(t,x_i(t))}\notag\\
    &\quad\quad\quad-  \frac{N(p+q-N+1))H(t,x_i(t))}{\partial_{z} H(t,x_i(t))}\notag
\\&= \frac{ (1-x_i(t)^2)\partial_{zz} H(t,x_i(t)) }{ \partial_{z} H(t,x_i(t))} +   (p-q)+(2(N-1)-(p+q))x_i(t).\notag
 \end{align}
  We now write $H(t,z)=H_i(t,z)(z-x_i(t))$ for $i=1,\ldots,N$. Hence, by(\ref{log-derivative-a2}),
 \begin{equation}\label{log-derivative-jacobi}
   \frac{\partial_{zz} H(t,x_i(t))}{\partial_{z} H(t,x_i(t))}=   \frac{2\partial_{z} H_i(t,x_i(t))}{ H_i(t,x_i(t))}=
  2 \sum_{j:j\ne i} \frac{1}{x_i(t)-x_j(t)}. \end{equation}
 As
 $$  \frac{1-x_i(t)^2}{x_i(t)-x_j(t)}=  \frac{1-x_i(t)x_j(t)}{x_i(t)-x_j(t)} - x_i(t),$$
 we obtain from (\ref{h4-jacobi}) and (\ref{log-derivative-jacobi}) that
 $$  x_i^\prime(t)=(p-q)-(p+q)x_i(t)	+2\sum_{j:\> j\neq i}\frac{1-x_i(t)x_j(t)}{x_i(t)-x_j(t)}$$
 as claimed.
 \end{proof}

 Notice that the  if-part in Theorem \ref{heat-equation-jacobi} also works for functions $H$ which satisfy
 the `` stationary heat equation with potential''  with  arbitary potentials.  On the other hand, we have the general assumption in the theorem
 that
 $H$ is a polynomial in $z$ of degree $N$ which is possible only for the particular potential there.

 \medskip

 Similar to Theorems \ref{expectation-a} and \ref{expectation-b} we now combine Theorem \ref{heat-equation-jacobi} with some result from
 \cite{V2}
in order to derive some expectations for multivariate Jacobi  processes on the compact alcoves $A_N$
as studied in \cite{De, RR, Do, AVW, V2, V3}. In view of the ODE \ref{ODE-jacobi}, the notations in \cite{V2}, and
Theorem \ref{expectation-jacobi} below, we here start with parameters
$\kappa>0$, $p,q>N-1+1/\kappa$
and define the associated Jacobi  processes  $(X_t:=X_{t,\kappa,p,q})_{t\ge0}$ on $A_N$ as the unique strong solutions
 of the SDEs
 \begin{equation}\label{SDE-jacobi}
   dX_{t}^i  =\sqrt{2(1-(X_{t}^i)^2)}\> d\tilde B_{t,i} +
\kappa\Bigl((p-q) -(p+q)X_{t}^i +
2\sum_{j: \> j\ne i}\frac{1-X_{t}^iX_{t}^j}{X_{t}^i-X_{t}^j}\Bigr)dt.
\end{equation}
for $i=1,\ldots,N$ with reflecting boundaries with some $N$-dimensional Brownian motion $(\tilde B_t)_{t\ge0}$ 
and reflecting boundaries.
Similar to the preceding sections we now regard $\kappa$ as an inverse temperature and study the
renormalized processes $(\tilde X_{t}:=\tilde X_{t,\kappa,p,q}:=X_{t/\kappa,\kappa,p,q})_{t\ge0}$  which then satisfy
\begin{equation}\label{SDE-jacobi-renormalized}
   d\tilde X_{t}^i  =\frac{\sqrt 2}{\sqrt\kappa } \sqrt{1-(\tilde X_{t}^i)^2}\> d\tilde B_{t,i} +
\Bigl((p-q) -(p+q)\tilde X_{t}^i +
2\sum_{j: \> j\ne i}\frac{1-\tilde X_{t}^i\tilde X_{t}^j}{\tilde X_{t}^i-\tilde X_{t}^j}\Bigr)dt,
\end{equation}
for $i=1,\ldots,N$, and which degenerate for $\kappa=\infty$ into the ODE (\ref{ODE-jacobi}).

Clearly, we can consider the Jacobi processes also for $N=1$. Here, we put  $\kappa=1$  and define a
Jacobi process $(Y_{t,p,q}:=X_{t,1,p,q})_{t\ge0}$ as diffusion on $[-1,1]$ with the generator
\begin{equation}Lf(z)=(1-z^2)f^{\prime\prime}(z)+((p-q)+(p+q)z)f^{\prime}(z).\end{equation}
With these notations:

\begin{theorem}\label{expectation-jacobi}
  Let $y\in [-1,1]$, $x\in A_N$, and $\kappa>0$, $p,q>N-1+1/\kappa$.
Let  $(\tilde X_{t,\kappa,p,q})_{t\ge0}$ be a renormalized Jacobi process on $A_N$
starting in $x$, and  let \\ $(Y_{t,p+N-1,q+N-1})_{t\ge0}$ a one-dimensional Jacobi process  on $[-1,1]$ starting in $y$ which is independent from
 $(\tilde X_{t,\kappa,p,q})_{t\ge0}$. Then
\begin{equation}\label{equ-expectation-jacobi}
\mathbb E\biggl(\prod_{i=1}^N (  Y_{t,p+N-1,q+N-1}- \tilde X_{t,\kappa,p,q}^i)\biggr)= e^{- N(p+q-N+1)t}\prod_{i=1}^N (y- x_i).
\end{equation}
\end{theorem}

\begin{proof}
As in the proof of Theorems \ref{expectation-a} and  \ref{expectation-b} we use the elementary
symmetric polynomials $e_l^N(x)$. Corollary 3.4 in \cite{V2} implies that
 for  $t\ge0$, $l=0,\ldots,N$, and each fixed starting point $x\in A_N$,
 the expectations
$\mathbb E( e_l^N(\tilde X_{t,\kappa,p,q}) )$ do not depend on $\kappa\in]0,\infty]$. This implies that
 $$E( e_l^N(\tilde X_{t,\kappa,p,q}))=  e_l^N(x(t))$$
 for the solution $x(t)$ of the ODE (\ref{ODE-jacobi}) with start in
 $x(0)=x$.
 Therefore, as in the proofs of  Theorems \ref{expectation-a} and  \ref{expectation-b}, we obtain  that
 \begin{align}\label{det-computation-jacobi}
\mathbb E\bigl(\prod_{i=1}^N (    Y_{t,p+N-1,q+N-1}- \tilde X_{t,\kappa,p,q}^i)\bigr)&=
\sum_{l=0}^N (-1)^l\mathbb E\bigl(e_l^N( \tilde X_{t,\kappa,p,q} )  Y_{t,p+N-1,q+N-1}^{N-l}\bigr)\notag\\
&=\sum_{l=0}^N (-1)^l\mathbb E(e_l^N( \tilde  X_{t,\kappa,p,q})) \cdot\mathbb E(Y_{t,p+N-1,q+N-1}^{N-l})\notag\\
&=\mathbb E\bigl(\prod_{l=1}^N ( Y_{t,p+N-1,q+N-1}- x_l(t))\bigr) .
\notag\end{align}
Theorem \ref{heat-equation-jacobi}, our choice of the parameters of our one-dimensional Jacobi process, and the Feynman-Kac theorem now imply that
$$\bigl(e^{N(p+q-N+1)t}\prod_{l=1}^N ( Y_{t,p+N-1,q+N-1}- x_l(t))\bigr)_{t\ge0}$$ is a martingale. This readily yields the claim.
\end{proof}

\section{The noncompact Jacobi case}

In this section we breifly consider a non-compact analogue of the ODEs and Jacobi processes on the compact alcoves $A_N$.
In the literature these processes are known also as Heckman-Opdam processes of type BC; see  \cite{Sch1, Sch2} for the basics and \cite{AVW, RV2, V3}
for particular topics. In order to point out the  connection with the preceding section, we study these processes in algebraic
coordinates like in  \cite{AVW, RV2, V3}
and not in trigonometric ones  in harmonic analysis in \cite{HO, HS, Sch1, Sch2}.
We again start with some ODEs which appear as freezing limits. More precisely, for an integer $N\ge1$ and  $p, q>N-1$ we consider the ODE
\begin{equation}\label{ODE_main-noncompact}
		\frac{d}{dt}x_i(t)
		=(q-p)+(q+p)x_i(t)
			+2\sum_{j: j\neq i}\frac{x_i(t)
			x_j(t)-1}{x_i(t)-x_j(t)} \quad (i=1,\dots,N).
\end{equation}
for 
$$x\in C_N:=\{x\in\mathbb R^N: \> 1\le x_1\le \ldots\le x_N\}.$$
Like in the preceding cases, we have the following result; see \cite{AVW}:

 \begin{theorem}\label{main-solutions-ode-noncompact}
          For each
		 each $x_0\in C_N$ the ODE \eqref{ODE_main-noncompact} has a unique 
		 solution $x(t)$ for $t\geq0$, i.e., there is a unique continuous function $x:[0,\infty)\to C_N$ with $x(0)=x_0$ such that
		  for $t>0$, $x(t)$ is in the interior of  $ C_N$ and  satisfies \eqref{ODE_main-noncompact}.
 \end{theorem}
 
 Please notice that the RHS of (\ref{ODE_main-noncompact}) is equal to  the  RHS of (\ref{ODE-jacobi}) up to the sign and the domain.
 While in the compact case in Section 4 we had asymptotic stationarity, we here have a hyperbolic growth, i.e., convergence results like
 Theorems \ref{final-convergence-a} and  \ref{final-convergence-b} are not available here.

 On the other hand, due to the connection of  (\ref{ODE_main-noncompact}) and (\ref{ODE-jacobi}), the results concerning the inverse heat equation can be
 extended to (\ref{ODE_main-noncompact})  by the computations in Section 4. We only remark that all moments of all Jacobi processes are again
 finite  which  is not clear a priori;
see Lemma 2.1 in \cite{RV2}.

\begin{theorem}\label{heat-equation-jacobi-noncompact}
  Let $x:=(x_1,\ldots,x_N):[0,\infty[\to C_N$ be a differentiable function.
      Then  $x(t)$ is a solution of (\ref{ODE_main-noncompact})
in the sense of Theorem \ref{main-solutions-ode-noncompact}
      if and only if the function $H(t,z):=\prod_{i=1}^N (z-x_i(t))$
      solves the inverse `` Jacobi-type heat equation with potential''
 \begin{equation}\label{heat-jacobi-noncompact}
   H_t=  -\Biggl( (z^2-1) H_{zz}+ \Bigl((2(N-1)-(p+q))z- (q-p)\Bigr) H_{z} \Biggr) +N(p+q-N+1)H.
   \end{equation}
 \end{theorem}

This result can be again used to compute some expectations. For this we follow the notations in Section 3 and \cite{AVW}
and define
the Heckman-Opdam processes $(X_t:=X_{t,\kappa,p,q})_{t\ge0}$ of type BC, i.e., the Jacobi processes in the noncompact setting, as diffusions
with the generators
\begin{equation}\label{generator-noncompact}
  L_kf(x):=\sum_{i=1}^N (x_i^2-1)f_{x_ix_i}(x)+
 \kappa \sum_{i=1}^N\Biggl((q-p) + (q+p)x_i
  +2\sum_{j: j\ne i} \frac{x_ix_j-1}{x_i-x_j}\Biggr)f_{x_i}(x)
\end{equation}
for $\kappa>0$, $p,q>N-1$ with reflecting boundaries. The renormalized processes
$(\tilde X_{t}:=\tilde X_{t,\kappa,p,q}:= X_{t/\kappa, \kappa,p,q})_{t\ge0}$ then can be seen  as the 
unique strong solutions of the SDEs
\begin{equation}\label{SDE-alcove-normalized-noncompact}
  d\tilde X_{t}^i =\frac{\sqrt 2}{\sqrt\kappa } \sqrt{(\tilde X_{t}^i)^2-1}\> d\tilde B_{t,i} 
+\Bigl((q-p) +(q+p)\tilde X_{t}^i +
2\sum_{j: j\ne i}\frac{\tilde X_{t}^i\tilde X_{t}^j-1}{\tilde X_{t}^i-\tilde X_{t}^j}\Bigr)dt
\end{equation}
for $  i=1,\ldots,N$. 
For $\kappa=\infty$, these SDEs degenerate to the ODEs (\ref{ODE_main-noncompact}).

Clearly, we can consider the Jacobi processes also for $N=1$. In this case we put  $\kappa=1$ and define a
Jacobi process $(Y_{t,p,q}:=X_{t,1,p,q})_{t\ge0}$ as diffusion on $[1,\infty[$ with the generator
$$Lf(z)=(z^2-1)f^{\prime\prime}(z)+((q-p)+(p+q)z)f^{\prime}(z).$$
With these notations we obtain as in the proof of Theorem \ref{expectation-jacobi}:

\begin{theorem}\label{expectation-jacobi-noncompact}
  Let $y\in [1,\infty[$, $x=(x_1,\ldots,x_N)\in C_N$, and $\kappa>0$, $p,q>N-1+1/\kappa$.
Let  $(\tilde X_{t,\kappa,p,q})_{t\ge0}$ be a renormalized Jacobi process on $C_N$
starting in $x$, and  $(Y_{t,p+N-1,q+N-1})_{t\ge0}$ a one-dimensional Jacobi process  on $[1,\infty[$ starting in $y$
    which is independent from
 $(\tilde X_{t,\kappa,p,q})_{t\ge0}$. Then
\begin{equation}\label{equ-expectation-jacobi-noncompact}
\mathbb E\bigl(\prod_{i=1}^N (  Y_{t,p+N-1,q+N-1}- \tilde X_{t,\kappa,p,q}^i)\bigr)= e^{N(p+q-N+1)t}\prod_{i=1}^N (y- x_i).
\end{equation}
\end{theorem}

\section{The torus case}

In this section we consider  Calogero-Moser-Sutherland models
with $N$  particles on the torus $\mathbb T:=\{z\in\mathbb C:\> |z|=1\}$; see \cite{LV}.
Due to computational problems on $\mathbb T$ w.r.t.~inequalities, we 
 begin with the associated diffusions 
 on  $\mathbb R^N$ with  $2\pi$-periodicity  such that the diffusions on $\mathbb T^N$ appear as
 images under the map $x \mapsto e^{ix}$ in all coordinates.
 In these trigonometric coordinates, we  start with the  Feller diffusions $(X_{t,k})_{t\ge0}$ on
 \begin{equation}\label{trig-torus-state-space}
   C_N := \{x\in \mathbb R^N: \, x_1\le x_2\le\ldots\le x_N\le x_1+2\pi\}.
 \end{equation}
for $k>0$
 which have the Heckman-Opdam Laplacians
  \begin{equation}
     L_kf(x)= \Delta f(x)+ k\sum_{j=1}^N \sum_{l\ne j}
    \cot\Bigl(\frac{x_j-x_l}{2}\Bigr)\frac{\partial}{\partial x_j}f(x)
   \end{equation}
 as generators  with reflecting boundaries; see  \cite{RR, RV2}.
 As before, we introduce the renormalized generators
 $\tilde L_k:=\frac{1}{k}  L_k$ and diffusions $(\tilde X_{t,k}:= X_{t/k,k})_{t\ge0}$
 which can be regarded as solutions of
the SDEs
     \begin{equation}\label{sde-torus-R}
       d\tilde X_{t,k}^j=\frac{\sqrt 2}{\sqrt k}dB_{t}^j+\sum_{l: l\ne j} \cot\Bigl(\frac{\tilde X_{t,k}^j-\tilde X_{t,k}^l}{2}\Bigr)dt
       \quad\quad(j=1,\ldots,N)
     \end{equation}
     with some $N$-dimensional Brownian motion $(B_{t})_{t\ge0}$. For $k=\infty$, this degenerates into the ODE
\begin{equation}\label{ODE-torus-R}
    x_j^\prime(t)= \sum_{l: \> l\ne j}\cot\Bigl(\frac{x_j(t)-x_l(t)}{2}\Bigr) \quad\quad(j=1,\ldots,N).
    \end{equation}
We have the following analogue of Theorem \ref{main-solutions-ode-jacobi}:

 \begin{theorem}\label{main-solutions-ode-torus}
     For  each $N\ge2$ and each starting point
		 each $x_0\in C_N$, \eqref{ODE-torus-R} has a unique 
		 solution $x(t)$ for  $t\geq0$ in the sense that
                  there  
		 is a unique continuous  $x:[0,\infty)\to C_N$ with $x(0)=x_0$ with 
		   $x(t)\in C_N\setminus\partial C_N$ for  $t>0$ such that \eqref{ODE-torus-R} holds for  $t>0$.
                  
		   Moreover, for all solutions of \eqref{ODE-torus-R},
                   \begin{equation}\label{ODE-torus-constant-mean}
                     x_1(t)+\ldots+x_N(t)=x_1(0)+\ldots+x_N(0) \quad\quad (t\ge0)\end{equation}
 and                  
 \begin{equation}\label{limit-torus1}
   \lim_{t\to\infty}x(t)=\Bigl(x_1,x_1+\frac{1}{N}2\pi,x_1+\frac{2}{N}2\pi,\ldots, x_1+\frac{N-1}{N}2\pi\Bigr)\in C_N \end{equation}
 with
 $$x_1:= \frac{1}{N}\Bigl( x_1(0)+\ldots+x_N(0) -(N-1)\pi\Bigr).$$
The points on the RHS of (\ref{limit-torus1}) are stationary for the ODE (\ref{ODE-torus-R}).
  \end{theorem}

 \begin{proof} (\ref{ODE-torus-constant-mean}) is obvious. For   (\ref{limit-torus1}) and the stationarity of points on the RHS
 of   (\ref{limit-torus1})  see \cite{RV2}. This and the methods in the appendix 
   of  \cite{AVW} on Theorem \ref{main-solutions-ode-jacobi} then yield the first statement in the theorem. We  omit the details.
\end{proof}   

 We have the following  order of convergence in (\ref{limit-torus1}):

 \begin{proposition}\label{decreasing-error-torus}
   Let  $x(t),\tilde x(t)$ be solutions of (\ref{ODE-torus-R}) with $x(0),\tilde x(0)\in C_N$ and
   $$x_1(0)+\ldots+x_N(0)=\tilde x_1(0)+\ldots+\tilde x_N(0).$$
Then for $t\ge0$,
  $$\Bigl\|x(t)-\tilde x(t)\Bigr\|\le  e^{- N t/2}\Bigl\|x(0)-\tilde x(0)\Bigr\|.$$ 
 \end{proposition}

 \begin{proof} We proceed as in the proof of Proposition
 \ref{decreasing-error-jacobi-stat-trigono}.
   Let $r_j(t):=x_j(t)-\tilde x_j(t)$ for $j=1,\ldots,N$, and $D(t):=\frac{1}{2}\sum_{j=1}^N r_j(t)^2$.
   We assume w.l.o.g.~that
   \begin{equation}\label{sum-zero-tortus}
     x_1(t)+\ldots+x_N(t)=\tilde x_1(t)+\ldots+\tilde x_N(t)=0 \quad\quad(t\ge0).\end{equation}
   Then, 
\begin{align}
  D^\prime(t) = &
  \sum_{j=1}^N\Bigl( x_j^\prime(t)- \tilde x_j^\prime(t)\Bigr)\Bigl( x_j(t)- \tilde x_j(t)\Bigr)\notag\\
  =& \sum_{l,j: \> l\ne j}\Bigl( x_j(t)- \tilde x_j(t)\Bigr)\Bigl(\cot\Bigl(\frac{x_j(t)-x_l(t)}{2}\Bigr)
-\cot\Bigl(\frac{\tilde x_j(t)-\tilde x_l(t)}{2}\Bigr)\Bigr)\notag\\
=& \frac{1}{2}\sum_{l,j: \> l\ne j}\Bigl((x_j(t)-x_l(t))-(\tilde x_j(t)-\tilde x_l(t))\Bigr)\cdot\notag\\
&\quad\quad\quad\quad\quad
\cdot\Biggl(\cot\Bigl(\frac{x_j(t)-x_l(t)}{2}\Bigr)
 -\cot\Bigl(\frac{\tilde x_j(t)-\tilde x_l(t)}{2}\Bigr)\Biggr).\notag
 \end{align}
As $x_j(t)-x_l(t)$ and $\tilde x_j(t)-\tilde x_l(t)$ have the same sign and are contained in $]-2\pi,2\pi[$, the mean value theorem yields that
$$     \cot\Bigl(\frac{x_j(t)-x_l(t)}{2}\Bigr)-\cot\Bigl(\frac{\tilde x_j(t)-\tilde x_l(t)}{2}\Bigr)
                   =-\frac{   (x_j(t)-x_l(t))-(\tilde x_j(t)-\tilde x_l(t))    }{2\sin^2(z_{j,l}(t)/2)}$$
                   with some $z_{j,l}(t)$ between $x_j(t)-x_l(t)$ and $\tilde x_j(t)-\tilde x_l(t)$.
Moreover, as by (\ref{sum-zero-tortus}),
$$\sum_{l,j: \> l\ne j}(x_j(t)-\tilde x_j(t))(x_l(t)-\tilde x_l(t))=- \sum_{j=1}^N(x_j(t)-\tilde x_j(t))^2,$$
          and as  $|\sin y|\le 1$ for $y\in\mathbb R$, we conclude that        
\begin{align}
 D^\prime(t) = &-\frac{1}{4}\sum_{l,j: \> l\ne j}\frac{\Bigl((x_j(t)-x_l(t))-(\tilde x_j(t)-\tilde x_l(t))\Bigr)^2}{\sin^2(z_{j,l}(t)/2)}\notag\\
 &\le -\frac{1}{4}    \sum_{l,j: \> l\ne j}\Bigl((x_j(t)-\tilde x_j(t))- (x_l(t)-\tilde x_l(t))\Bigr)^2\notag\\
&=-\frac{1}{4}\Biggl( 2(N-1)\sum_{j=1}^N(x_j(t)-\tilde x_j(t))^2 -2\sum_{l,j: \> l\ne j}(x_j(t)-\tilde x_j(t))(x_l(t)-\tilde x_l(t))\Biggr)\notag\\
&=-N \cdot D(t).\notag\end{align}
 The lemma of Gronwall now implies the claim.
 \end{proof}

 It does not seem to be possible to state an analogue of the connection between the ODEs and the associated  inverse heat equations
 in the preceding sections for the ODEs (\ref{ODE-torus-R}) due to their trigonometric form. To obtain such a connection,
 we 
 transfer all data from $\mathbb R^N$ to  $\mathbb T^N$ via $w_j:=e^{ix_j}$ for $j=1,\ldots,N$, or for short,
  $w:=e^{ix}\in \mathbb T^N$. Then in $w$-coordinates,  $L_k$ is given by 
    \begin{align}\label{Lap-Vinet-op}
      H_k =&  -\sum_{j=1}^N\Bigl( w_j\frac{\partial}{\partial w_j}\Bigr)^2-
      k\sum_{j=1}^N \sum_{l:\> l\ne j}\frac{w_j+w_l}{w_j-w_l}\cdot w_j\frac{\partial}{\partial w_j}\\
      =& -\sum_{j=1}^N w_j^2 \frac{\partial^2}{\partial w_j^2} -(1-k(N-1))\sum_{j=1}^N w_j\frac{\partial}{\partial w_j}
-2k\sum_{j=1}^N \sum_{l:l\ne j}
\frac{ w_j^2}{w_j-w_l}\frac{\partial}{\partial w_j}\notag\\
=& -\sum_{j=1}^N w_j^2 \frac{\partial^2}{\partial w_j^2} -(1+k(N-1))\sum_{j=1}^N w_j\frac{\partial}{\partial w_j}
-2k\sum_{j=1}^N \sum_{l:l\ne j}
\frac{ w_jw_l}{w_j-w_l}\frac{\partial}{\partial w_j}\notag
    \end{align}
   where $H_k$ is applied to  permutation invariant functions in
   $C^2(\mathbb T^N)$; see  \cite{LV, RV2}.
We now consider
  the renormalized operators $\tilde  H_k:=\frac{1}{k} H_k$ and the associated diffusions
  $(\tilde Z_{t,k}:=e^{i\widetilde X_{t,k}})_{t\ge0}$
  on $C_N^{\mathbb T};=\{e^{ix}:\> x\in C_N\}$  for $k\in]0,\infty[$. Clearly, this  also works for $k=\infty$. Hence,
                   in $w$-coordinates, the ODE (\ref{ODE-torus-R}) 
                   has the form
    \begin{equation}\label{ODE-torus}
      w_j^\prime(t)= -(N-1) w_j(t)
-2\sum_{l:l\ne j}
\frac{ w_j(t)w_l(t)}{w_j(t)-w_l(t)}
\quad\quad(j=1,\ldots,N).
    \end{equation}

\begin{remark}  Besides the generators in  (\ref{Lap-Vinet-op}), also the operators
  \begin{equation}
    D_k:= \sum_{j=1}^N w_j^2 \frac{\partial^2}{\partial w_j^2}     +2k\sum_{j=1}^N \sum_{l: \> l\ne j}
    \frac{ w_j^2}{w_j-w_l}\frac{\partial}{\partial w_j}
    \end{equation}
appear in the literature; see e.g.~\cite{F, St, OO}.
Here
$-H_k=D_k+E_k$ with the Euler operator 
$E_k:=(1-k(N-1))\sum_{j=1}^N w_j\frac{\partial}{\partial w_j}$
which commutes with $D_k$.
\end{remark}
     
We now turn to the connection of (\ref{ODE-torus}) to  inverse heat equations.

\begin{theorem}\label{heat-equation-torus}
  Let $w:=(w_1,\ldots,w_N):[0,\infty[\to C_N^{\mathbb T}$ be a differentiable function.
      Then  $w$ is a solution of (\ref{ODE-torus})
      if and only if the function $H(t,z):=\prod_{j=1}^N (z-w_j(t))$
      solves the ``inverse heat equation''
 \begin{equation}\label{heat-torus}
   H_t= z^2 H_{zz}-(N-1)z H_{z}.
   \end{equation}
 \end{theorem}

 \begin{proof}
   Assume  that $w$ satisfies  (\ref{ODE-torus}).
   Let $H$ be as  in the theorem, and, for $l\ne j$,
   $$  H_j(t,z):=H(t,w)/(z-w_j(t)), \quad
     H_{j,l}(t,z):=H(t,z)/((z-w_j(t))(z-w_l(t))). $$
   Then by (\ref{ODE-torus}), 
   (\ref{two-to-one}),  and (\ref{h1-jacobi}),
\begin{align}\label{h2-torus}
&  \partial_t H(t,z)=   -\sum_{j=1}^N w_i^\prime(t) \> H_j(t,z)\\
     =& (N-1)\sum_{j=1}^N w_j(t)H_j(t,z)+2\sum_{l,j: \> l\ne j} \frac{w_j(t)w_l(t)}{w_j(t)-w_l(t)}\> H_j(t,z)\notag\\
     =&(N-1)\sum_{j=1}^N (z-(z-w_j(t))H_j(t,z)+\sum_{l,j: \> l\ne j} \frac{w_j(t)w_l(t)}{w_j(t)-w_l(t)} ( H_j(t,z)-H_l(t,z))\notag\\
  =&(N-1)\bigl(   z\partial_z H(t,z)-N \cdot H(t,z)\bigr) +\sum_{l,j: \> l\ne j} w_j(t)w_l(t) H_{l,j}(t,z).\notag\end{align}
  Hence, by (\ref{h3-jacobi})
\begin{align}\label{h3-torus}
  \partial_t H(t,z)=&\bigl( (N-1)  z\partial_z H(t,z)-(N-1)N \cdot H(t,z)\bigr)\notag\\ &\quad+
\bigl( (N-1)N \cdot H(t,z)+z^2 \partial_{zz} H(t,z)-2(N-1)  z\partial_z H(t,z)\bigr)\notag\\
=&z^2 \partial_{zz} H(t,z)-(N-1)z\partial_z H(t,z)
\end{align}
 as claimed.

Now assume that (\ref{heat-torus}) holds. As in the proof of Theorem \ref{heat-equation-a}, we obtain from
(\ref{implicit}) and (\ref{heat-torus}) that for $j=1,\ldots,N$,
  $$  w_j^\prime(t)=- \frac{\partial_{t} H(t,w_j(t))}{\partial_{z} H(t,w_j(t))} 
    =
    \frac{ -w_j(t)^2\partial_{zz} H(t,w_j(t))}{\partial_{z} H(t,w_j(t))}+ (N-1) w_j(t).$$
Hence, by  (\ref{log-derivative-jacobi}), 
$$  w_j^\prime(t)=- 2w_j(t)^2\sum_{l: \> l\ne j}   \frac{1}{w_j(t)-w_l(t)}+ (N-1) w_j(t),$$
which is just  (\ref{ODE-torus}) as claimed.
 \end{proof}

As in the preceding sections,
Theorem \ref{heat-equation-torus} leads to some martingales. For this consider the process $(Y_t:=exp( i\sqrt 2 \cdot \tilde B_t +Nt))_{t\ge0}$
for 
some one-dimensional Brownian motion
 $(\tilde B_t)_{t\ge0}$ which satisfies
 $$dY_t=  i\sqrt 2 \cdot Y_t \> d \tilde B_t +(N-1)Y_t\> dt.$$
Dynkin's formula and a  comparison of this SDE with the inverse heat equation (\ref{heat-torus})  imply:

 \begin{corollary}\label{martingale-torus}
   Let  $w:[0,\infty[\to C_N^{\mathbb T}$ be any solution  of (\ref{ODE-torus}). Then for any one-dimensional Brownian motion
       $(\tilde B_t)_{t\ge0}$, the process
$$\Bigl( \prod_{j=1}^N (e^{Nt}e^{ i\sqrt 2 \cdot \tilde B_t } -w_j(t))\Bigr)_{t\ge0}$$
   is a martingale.    
 \end{corollary}

 Similar to Theorems \ref{expectation-a}, \ref{expectation-b}, and \ref{expectation-jacobi}, this can be used to compute some expectations:

\begin{theorem}\label{expectation-torus}
  Let $k>0$, $y\in \mathbb R$, and  $w_0=(w_{0,1},\ldots w_{0,N} )\in C_N^{\mathbb T}$.
Let  
  $(\tilde Z_{t,k}:=e^{i\widetilde X_{t,k}})_{t\ge0}$ be a renormalized particle diffusion on   $C_N^{\mathbb T}$ as described  above starting in $w_0$, and let
 $(\tilde B_t)_{t\ge0}$ be one-dimensional Brownian motion starting in $y$ which is independent from  $(\tilde Z_{t,k})_{t\ge0}$.
Then
\begin{equation}\label{equ-expectation-torus}
\mathbb E\biggl(\prod_{j=1}^N (  e^{Nt}e^{ i\sqrt 2 \cdot \tilde B_t } - e^{t/k}\tilde Z_{t,k}^j )\biggr)= \prod_{j=1}^N (e^{ i\sqrt 2 \cdot y}-w_{0,j}).
\end{equation}
\end{theorem}

\begin{proof}
We again use the elementary
symmetric polynomials $e_l^N(x)$. By Corollaries 3.5 and 3.6  in \cite{RV2} we know that
\begin{equation}\label{expectation-elementary-torus}
  \mathbb E\bigl(e_l^N(\tilde Z_{t,k})\bigr)= e^{-l(N-l+1/k)t} e_l^N(w_0)
  \end{equation}
for $t\ge0$, $l=0,\ldots,N$, and $k\in ]0,\infty]$, where for $k=\infty$, $\tilde Z_{t,k}=w(t)$ is a deterministic solution of the ODE
   (\ref{ODE-torus}).
   Hence, by (\ref{expectation-elementary-torus}) and Corollary \ref{expectation-torus},
   \begin{align}\label{det-computation-torus}
     \mathbb E\bigl(\prod_{i=1}^N ( e^{Nt}e^{ i\sqrt 2 \cdot \tilde B_t } &- e^{t/k}\tilde Z_{t,k}^j )\biggr)=
     \sum_{l=0}^N (-1)^l\mathbb E\bigl( e^{(N-l)Nt}e^{ i\sqrt 2\cdot(N-l) \cdot \tilde B_t }  \cdot e^{lt/k}   e_l^N( \tilde Z_{t,k})\bigr)\notag\\
     &=\sum_{l=0}^N (-1)^l\mathbb E\bigl( e^{(N-l)Nt}e^{ i\sqrt 2\cdot(N-l) \cdot \tilde B_t } \bigr)
     \cdot\mathbb E\bigl( e^{lt/k}   e_l^N( \tilde Z_{t,k})\bigr)\notag\\
 &=\sum_{l=0}^N (-1)^l\mathbb E\bigl( e^{(N-l)Nt}e^{ i\sqrt 2\cdot(N-l) \cdot \tilde B_t } \bigr)   \cdot e_l^N( w(t))\bigr)\notag\\
 &=\mathbb E\bigl(\prod_{i=1}^N ( e^{Nt}e^{ i\sqrt 2 \cdot \tilde B_t } - w_j(t))\bigr)
     = \prod_{j=1}^N (e^{ i\sqrt 2 \cdot y}-w_{0,j})
     \end{align}
     as claimed.
   \end{proof}

\begin{remark} Clearly, (\ref{equ-expectation-torus}) can be derived without
  Corollary \ref{expectation-torus} and thus Theorem \ref{heat-equation-torus}. In fact, one just has to insert 
  (\ref{expectation-elementary-torus}) and the expectations of complex geometric Brownian motions in the second line of
  (\ref{det-computation-torus}). On the other hand, the proof above explains, why we have to choose
  the process $(e^{Nt}e^{ i\sqrt 2 \cdot \tilde B_t })_{t\ge0}$ in (\ref{equ-expectation-torus}) in order to get  expectations which are more or less
  independent from $k$ and $t$.

  The same alternative approach to proofs of Theorems \ref{expectation-a}, \ref{expectation-b}, and \ref{expectation-jacobi}
  is also possible without using the corresponding inverse heat equations where then  (\ref{expectation-elementary-torus})
  has to be replaced by corresponding identities for the expectations of the elementary symmetric polynomials
  of the corresponding multivariate Bessel and Jacobi processes from \cite{KVW, V2}. However, in these cases, the computation would be more involved
  as some of these expectations are determined via recursive identities.
     \end{remark}

\begin{remark} Similar to the connection of the noncompact and compact Jacobi case in Section 4 and 5, there exists a noncompact
hyperbolic  analogue of the torus case above; see \cite{HO, HS} for the algebraic background and \cite{Sch1,Sch2, RV2} for probabilistic aspects.
The starting point here is the Heckman-Opdam
Laplacian
  \begin{equation}
    L_kf(x)= \Delta f(x)+ k\sum_{j=1}^N \sum_{l\ne j}
    \coth\Bigl(\frac{x_j-x_l}{2}\Bigr)\frac{\partial}{\partial x_j}f(x)
   \end{equation}
for  $k\in[0,\infty[$ on the Weyl chamber
$C_N^A:=\{x\in \mathbb R^N:\> x_1\ge x_2\ge\ldots \ge x_N\}$. Similar to Section 5, one obtains analogues of Theorems \ref{heat-equation-torus}
and \ref{expectation-torus}.   \end{remark}


\begin{thebibliography}{999}

  \bibitem[AGZ]{AGZ} G.W. Anderson, A. Guionnet, O. Zeitouni, 
  An Introduction to Random Matrices. Cambridge University Press, 2010.


\bibitem[AKM1]{AKM1} S. Andraus, M. Katori, S. Miyashita, Interacting particles on the line 
and Dunkl intertwining operator of type $A$: Application to the freezing regime. 
\textit{J. Phys. A: Math. Theor. } 45  (2012) 395201.

\bibitem[AKM2]{AKM2} S. Andraus, M. Katori, S. Miyashita, 
Two limiting regimes of interacting Bessel processes. 
 \textit{J. Phys. A: Math. Theor. } 47  (2014) 235201.

\bibitem[AV1]{AV1} S. Andraus, M. Voit, Limit theorems
 for multivariate Bessel processes in the freezing regime. \textit{Stoch. Proc. Appl. }
 129 (2019), 4771-4790.

\bibitem[AV2]{AV2} S. Andraus, M. Voit, Central limit theorems for multivariate Bessel processes
 in the freezing regime II: The covariance matrices.
\textit{J. Approx. Theory}  246 (2019), 65-84. 

\bibitem[A]{A} J.-P. Anker,  An introduction to Dunkl theory and its analytic aspects.
  In: G. Filipuk et al.. Analytic, Algebraic and Geometric Aspects of Differential Equations,
  Birkh{\"a}user, pp.3-58, 2017.

\bibitem[AVW]{AVW} M. Auer, M. Voit, J. Woerner, Wigner- and Marchenko-Pastur-type limit
  theorems for Jacobi processes. Peprint 2022, arXiv:2203.07797.

  
 \bibitem[BF]{BF} T.H. Baker, P.J. Forrester, The Calogero-Sutherland model and generalized classical
polynomials. \textit{Comm. Math. Phys.} 188 (1997), 175--216.

  \bibitem[CL]{CL}  E. Cepa and D. Lepingle, Diffusing particles with electrostatic repulsion. \textit{ Probab. Theory
    Relat. Fields} 107 (1997), 429–449.

    
  \bibitem[CGY]{CGY} O. Chybiryakov, L. Gallardo, M. Yor, Dunkl processes and their radial parts relative to a root system. In:
    P. Graczyk et al. (eds.), Harmonic and stochastic analysis of Dunkl processes. Hermann, Paris 2008.
    
\bibitem[CSL]{CSL} G. Csordas, W. Smith, D.H. Lehmer, Lehmer pairs of zeros, the de Bruijn-Newman constant $\Lambda$, and the Riemann hypothesis.
  \textit{Constr. Approx.} 10 (1997), 107-129.
 


 \bibitem[De]{De} N. Demni,
$\beta$-Jacobi processes.
   \textit{Adv. Pure Appl. Math.} 1 (2010), 325-344.

\bibitem[DG]{DG} P. Diaconis, A. Gamburd, 
Random matrices, magic squares and matching polynomials. 
\textit{Electron. J. Combin.} 11 (2004/06), no. 2, Research Paper 2, 26 pp.. 

   

\bibitem[DV]{DV} J.F. van Diejen, L. Vinet, Calogero-Sutherland-Moser Models.
 CRM Series in Mathematical Physics, Springer, Berlin, 2000.


 \bibitem[Do]{Do} Y. Doumerc, Matrix Jacobi process, Ph.D. Thesis, Paul Sabatier University, 2005.



\bibitem[DE]{DE2} I. Dumitriu, A. Edelman, Eigenvalues of Hermite and Laguerre ensembles: large beta asymptotics,
  \textit{Ann. Inst. Henri Poincare (B)} 41 (2005), 1083-1099.
  
\bibitem[Du]{Du} C. Dunkl, Dunkl Operators and Related Special Functions.
  In: Encyclopedia of Special Functions, Part II: Multivariable Special Functions, eds. T.H. Koornwinder, J.V. Stokman,
  Cambridge University Press, Cambridge, 2021.

  

  \bibitem[F]{F}  P. Forrester, Log Gases and Random Matrices, London Mathematical Society, London, 2010.


\bibitem[FG]{FG} P. Forrester, A. Gamburd, Counting formulas associated with some random matrix averages. 
\textit{J. Combin. Theory A} 113 (2006), 934-951.


\bibitem[GM]{GM} P.  Graczyk, J. Malecki, Strong solutions of non-colliding particle systems.
\textit{ Electron. J. Probab.} 19 (2014), 21 pp.

\bibitem[HHJK]{HHJK} B. Hall, C.-W. Ho, J. Jalowy, Z. Kabluchko, The heat flow, GAF, and  $SL(2; \mathbb R)$.
Preprint,  	arXiv:2308.11685. 


  \bibitem[HO]{HO} G. Heckman, E. Opdam, Jacobi polynomials and hypergeometric functions associated with root systems.
  In: Encyclopedia of Special Functions, Part II: Multivariable Special Functions, eds. T.H. Koornwinder, J.V. Stokman,
  Cambridge University Press, Cambridge, 2021.

\bibitem[HeS]{HeS} G. Heckman, H. Schlichtkrull, 
Harmonic Analysis and Special Functions on Symmetric Spaces, Part I.  
Perspectives in Mathematics, Vol. 16, Academic Press, 1994.



\bibitem[HV]{HV} K. Hermann, M. Voit, Limit theorems for Jacobi ensembles with large parameters.
  \textit{Tunisian J. Math.} 3-4 (2021), 843--860.

  \bibitem[HS]{HS} M.W. Hirsch, S. Smale, Differential Equations, Dynamical Systems, and Linear Algebra.
    Academic Press, San Diego, CA, 1974.

 \bibitem[HW]{HW} D. Hobson, W. Werner, Non-colliding Brownian motions on the circle.
\textit{ Bull. London Math. Soc.} 28 (1996), 643--650.   
    
\bibitem[Ka]{Ka} M. Katori, 
Bessel processes, Schramm-Loewner evolution, and the Dyson model.
 Springer, Singapore (2015).
  
    
    
\bibitem[KVW]{KVW} M. Kornyik, M. Voit, J. Woerner, Some martingales associated with
  multivariate Bessel  processes. \textit{Acta Math. Hungarica}  163  (2021),  194-212.

 

\bibitem[LV]{LV} L. Lapointe, L. Vinet, Exact operator solution of the Calogero-Sutherland model.
  \textit{Comm. Math. Phys.} 178 (1996), 425-452.
   
\bibitem[OO]{OO}  A. Okounkov, G. Olshanski, Asymptotics of Jack polynomials as the number of variables goes to infinity. 
 \textit{ Int. Math. Res. Not.} 13  (1998), 641–682.


    
    \bibitem[RR]{RR}  H. Remling,  M. R\"osler,
  The heat semigroup in the compact Heckman-Opdam setting and the Segal-Bargmann transform. 
 \textit{ Int. Math. Res. Not.} 2011, No. 18, 4200-4225.


  \bibitem[RY]{RY} D. Revuz, M. Yor, Continuous Martingales and Brownian Motion, Springer Verlag, Berlin, Heidelberg, 1991.

  \bibitem[RT]{RT} B. Rodgers, T. Tao, The de Brujin-Newman constant is nonnegative. \textit{Forum Math. Pi} 8, e6, doi:10.1017/fmp.2020.6.

    \bibitem[RW]{RW} L.C.G.~Rogers, D.~Williams, Diffusions, Markov Processes and Martingales, Vol. 1 Foundations.
Cambridge University Press, Cambridge, 2000.


\bibitem[R1]{R1} M. R\"osler,
Generalized Hermite polynomials and the heat equation for Dunkl operators.
\textit{Comm. Math. Phys.} 192 (1998),  519-542.


\bibitem[R2]{R2} M. R\"osler, Dunkl operators: Theory and applications.
In: Orthogonal polynomials and special functions, Leuven 2002, \textit{Lecture Notes in Math.} 1817 (2003), 93--135.

\bibitem[RV1]{RV1} M. R\"osler, M. Voit, Markov processes related with Dunkl operators.
  \textit{Adv. Appl. Math.}  21 (1998) 575-643.

  \bibitem[RV2]{RV2} M. R\"osler, M. Voit, Elementary symmetric polynomials and martingales for Heckman-Opdam processes.
\textit{ Contemp. Math.} 780 (2022), 243-262.


\bibitem[Sch1]{Sch1} B. Schapira,
The Heckman-Opdam Markov processes.
 \textit{Probab. Theory Rel. Fields} 138 (2007), 
 495-519.

\bibitem[Sch2]{Sch2} B. Schapira, Contribution to the hypergeometric function theory of Heckman and Opdam: 
  sharp estimates, Schwarz space, heat kernel.  \textit{Geom. Funct. Anal.} 18 (2008), 222-250.


\bibitem[St]{St} R.P. Stanley, Some combinatorial properties 
  of Jack symmetric functions. \textit{Adv. Math.} 77 (1989), 76-115.
  

 

\bibitem[S]{S} G. Szeg{\"o}, Orthogonal Polynomials. 
Colloquium Publications (American Mathematical Society), Providence, 1939.


\bibitem[V1]{V} M. Voit, Central limit theorems 
  for multivariate Bessel processes in the freezing regime.\textit{ J. Approx. Theory } 239 (2019), 210--231.

   \bibitem[V2]{V2} M. Voit, Some martingales associated with multivariate Jacobi processes and Aomoto's Selberg integral.
     \textit{Indag. Math.} 31 (2020),  398-410.

   \bibitem[V3]{V3} M. Voit,   Freezing limits for  Calogero-Moser-Sutherland  particle models.  \textit{Studies Appl. Math.} 151 (2023),
 1230-1281.


     
 \bibitem[VW1]{VW1} M. Voit, J.H.C. Woerner, Functional 
 central limit theorems for multivariate Bessel processes in the
 freezing regime. \textit{ Stoch. Anal. Appl.} 39 (2021), 136-156. 

\bibitem[VW2]{VW2} M. Voit, J.H.C. Woerner,  Limit theorems for Bessel and Dunkl processes of large dimensions and free convolutions.
  \textit{Stoch. Proc. Appl.} 143 (2022), 207-253.

 \bibitem[VW3]{VW3} M. Voit, J.H.C. Woerner, The differential equations associated with Calogero-Moser-Sutherland
  particle models in the freezing regime. \textit{Hokkaido Math. J.} 51 (2022), 153--174.


\end{thebibliography}
\end{document}